\documentclass[18pt]{siamart220329}

\usepackage[english]{babel}
\usepackage{color}
\usepackage[letterpaper,top=2cm,bottom=2cm,left=3cm,right=3cm,marginparwidth=1.75cm]{geometry}
\usepackage{amsmath}
\usepackage{amsfonts}
\usepackage{graphicx}
\usepackage{todonotes}
\usepackage{algorithm}
\usepackage{algpseudocode}
\usepackage{comment}
\usepackage[normalem]{ulem}
\usepackage{multirow}
\usepackage{booktabs}
\usepackage{indentfirst}
\usepackage{caption}
\usepackage{subcaption}

\newcommand{\F}{\mathcal{F}}
\newcommand{\X}{\mathcal{X}}

\newcommand{\K}{\mathcal{K}}

\newcommand{\eq}[1]{\begin{equation}\label{#1}}
\newcommand{\en}{\end{equation}}
\def\up#1{^{({#1})}} %

\title{Anderson Acceleration with Truncated Gram-Schmidt}
\author{Ziyuan Tang\thanks{Department of Computer Science and Engineering, University of Minnesota, Minneapolis (\email{tang0389@umn.edu}, \email{saad@umn.edu}). The research of Tang and Saad is supported by the NSF award DMS 2208456.}
\and Tianshi Xu\thanks{Department of Mathematics, Emory University, Atlanta, GA 30322 (\email{tianshi.xu@emory.edu},
\email{yxi26@emory.edu}). The research of Xi is supported by NSF award  DMS 2208412.}
\and Huan He\thanks{Work done in Department of Computer Science, Emory University, Atlanta, GA 30322 (\email{hehuannb@gmail.com})}
\and Yousef Saad\footnotemark[1]
\and Yuanzhe Xi\footnotemark[2]}

\headers{Anderson Acceleration with Truncated Gram-Schmdit}{Z. Tang, T. Xu, H. He, Y. Saad, and Y. Xi}

\begin{document}
\maketitle

\begin{abstract}
Anderson Acceleration (AA) is a popular algorithm designed to  enhance 
the convergence of fixed-point iterations.
In this paper, we introduce a variant of AA based on a Truncated
Gram-Schmidt process (AATGS) which has a few advantages over the classical AA. 
In particular, an attractive
feature of AATGS is that its iterates obey a three-term recurrence in the situation
when it is applied to solving  symmetric linear problems and this can  lead to a
considerable reduction of memory and computational costs.
We analyze the convergence of AATGS in
both full-depth and limited-depth scenarios and establish its
equivalence to the classical AA in the linear case.
We also report on the effectiveness of AATGS  through a set of numerical
experiments, ranging from solving nonlinear partial differential
equations to tackling nonlinear optimization problems. In particular, the performance of the
method is compared with that of the classical AA algorithms.
\end{abstract}




\section{Introduction and Motivation}
\label{sec:intro}
This paper  considers  numerical schemes for solving the nonlinear system of equations
\begin{equation}
f(x) = 0, 
\label{eq:problem}
\end{equation}
where $f$ is a continuously differentiable mapping from $\mathbb{R}^n $ to $\mathbb{R}^n$.
Problem \eqref{eq:problem} can be reformulated as an equivalent fixed point problem
\begin{equation}
x = g(x),
\label{eq:fp}
\end{equation}
for a suitable mapping $g$ from $\mathbb{R}^n$ to $\mathbb{R}^n$. For
example, we can set $g(x) = x + \beta f(x)$ for some nonzero scalar
$\beta$.  When the {fixed point iteration}, i.e., the sequence
generated by $x_{j+1} = g (x_j)$, converges to the fixed point of
\eqref{eq:fp} then this limit is
a solution to the problem \eqref{eq:problem}. However, the
fixed-point iteration can be slow or it can diverge
and therefore acceleration methods are often invoked to improve or establish convergence.
Anderson Acceleration (AA) \cite{Anderson65}, which is equivalent to
the DIIS method - or Pulay Mixing \cite{pul80,Pulay-DIIS} in quantum
chemistry, is a popular acceleration technique that has been developed for this purpose. 
AA  has found extensive applications in scientific computing and, more recently, in  machine learning
\cite{doi:10.1137/20M132938X,gdaam,10.5555/3524938.3525552,10.1145/3197517.3201290,doi:10.1137/18M1206151,homer11,10.1007/s10915-021-01548-2,wei2021stochastic}.

If the j-th iterate is denoted by $x_j$ and if we set $f_j\equiv f(x_j)$, then AA
starts with an initial $x_0$ and defines $x_1=g(x_0)=x_0+\beta_0 f_0$,
where $\beta_0 >0$ is a parameter. Let ${m}_j = \min \{m,j\}$ and
$j_m = \max \{0,j-m\}\equiv j-m_j$  and assume that the most recent ${m}_j$ iterates are
saved at each step. At step $j$, we define the matrices of differences:
\begin{equation}
\label{eq:dfdx}
\mathcal{X}_j=[\Delta x_{j_m}\;\ldots\;\Delta x_{j-1}]\in \mathbb{R}^{n\times {m}_j},
\qquad
\mathcal{F}_j=[\Delta f_{j_m}\;\ldots\;\Delta f_{j-1}]\in \mathbb{R}^{n\times {m}_j},
\end{equation}
where  $\Delta x_i:=x_{i+1}-x_i$ and $\Delta f_i:=f_{i+1} - f_{i}$. 
Then AA defines the next iterate as follows: 
\begin{align} 
x_{j+1} &=  x_j+\beta_j f_j -(\X_j+\beta_j \F_j) \theta_{j}  \quad \mbox{where:}  \label{eq:AA}  \\
\theta_{j} &= \text{argmin}_{\theta \in \mathbb R^{{m}_j}}\| f_j - \F_j \theta \|_2 .  
 \label{eq:thetaj}
\end{align}
Note that $x_{j+1}$ can be expressed with the  help of intermediate vectors:
\begin{equation}
\label{eq:AA1} 
\bar x_j = x_j-\mathcal{X}_j  \theta_{j} ,\quad 
\bar f_j = f_j-\mathcal{F}_j \theta_{j} , \quad 
x_{j+1} =\bar x_j + \beta_j \bar f_j .
\end{equation}
AA is closely related to Broyden's multi-secant type methods. This
connection was initially revealed in \cite{eyert:acceleration96} and
further discussed in \cite{FangSaad07}. Essentially, AA acts as a
`block version' of Broyden's second update method where an
update of rank ${m}_j$ is applied at each step, instead of 
the  traditional rank $1$
update. Note that the AA scheme just discussed retains $m_j$ past iterates 
where $m$ is often called the \textit{window size}, or sometimes \textit{depth},
of the AA procedure in the literature. In  subsequent sections, we will refer to this scheme as
AA(m). Retaining and using all past iterates is equivalent to setting $m=\infty$ in the
procedure and so it will be often denoted by AA($\infty$).
This is often referred to as the \textit{full-depth} Anderson Acceleration, while when
$m < \infty$, AA(m) is known as a  \textit{limited-depth}, \textit{windowed} or \textit{truncated}
version of AA.

The study of the convergence of AA has been an active research area in
recent years. It was shown in \cite{homer11} that the full-depth
AA$(\infty)$ applied to $g(x)=Gx+b$ is ``essentially equivalent" to
the GMRES method \cite{Saad-Schultz-GMRES} applied to $(I-G)x=b$ when $I-G$ is nonsingular and
the linear residuals are strictly decreasing in the norm. Under these
  assumptions, the iterate $x_{j}$ returned by AA$(\infty)$ at step
  $j$ is equal to $Gx_{j-1}^{GMRES}+b$ where $x_{j-1}^{GMRES}$ is the
  iterate returned by GMRES(j-1) with the same initial guess $x_0$.
The first rigorous convergence analysis of AA(m) for contractive fixed
point mappings was conducted in \cite{Kelley15} where the authors
prove the q-linear convergence of the residuals for linear problems
and the local r-linear convergence for nonlinear problems when the
coefficients in the linear combination remain bounded. In addition,
they also prove the q-linear convergence of the residuals for AA(1)
separately. These convergence results show that the convergence rate
of AA(m) is not worse than that of the underlying fixed point
iteration.  The explicit improvement of AA(m) over the underlying
fixed point iteration at each step is studied in \cite{sara20} where
the authors show that AA(m) can improve the convergence rate to first
order by a factor $\tau_j\leq 1$ that is equal to the ratio of
$\Vert f_j-\mathcal{F}_j\theta_j\Vert_2$ to $\Vert
f_j\Vert_2$. They also point out that although AA(m) can
increase the radius of convergence, AA(m) typically fails to improve
the convergence in quadratically converging fixed point
iterations. The asymptotic convergence analysis of AA(m) is conducted
in \cite{hans22}, where the authors show that the r-linear convergence
factor strongly depends on the initial condition for the r-linearly
convergent AA(m) sequence and the coefficients $\theta_j$ do not
converge but oscillate as the sequence converges. The one-step
convergence analysis of inexact AA(m) with a potentially
non-contractive mapping is conducted in \cite{fei22}. The convergence
rate of AA(m) on superlinearly and sublinearly converging fixed point
iterations has recently been studied in \cite{leo23}.
Recent work has also addressed the numerical stability of AA(m). The article
\cite{de2021anderson}  showed some interesting theoretical results
for AA(1) for linear problems, and numerically studied the least-squares
problem in AA(m). The paper emphasized that the robustness of least-squares
solution techniques like those based on the 
QR factorization can ensure small backward errors and
accurate results without the need for regularization.  
A comprehensive analysis of backward stability for approximate least-squares
solves in AA for linear problems can also be found in \cite{lupo2022anderson},
where rigorous theoretical bounds are exploited   to minimize computational
costs.

While recent studies have concentrated on the convergence of AA
and on improving its convergence properties, relatively little attention has been
devoted to reducing its memory usage.  One of the goals of this paper is to address this issue.
The paper develops a variant of AA that can exploit the symmetry (or near symmetry) of
the Jacobian of the function $f$. In doing so, the iterates will obey short-term
update expressions akin to those of the
Conjugate Gradient or Conjugate Residual methods.  The end result is a
substantial reduction in memory and computational costs when solving large-scale
nonlinear equations or optimization problems.  Short-term recurrences often lead
to numerical instabilities, and thus the proposed algorithm may
encounter numerical issues in some situations. 
To circumvent this problem we introduce a restarting strategy that aims at
monitoring the growth of floating point errors.

The remaining sections are organized as follows.  AATGS is introduced
in Section \ref{sec:AATGS} which also presents a convergence analysis.
The restarting strategy is discussed in Section \ref{sec:restarting}
and numerical experiments are provided in Section
\ref{sec:exp}. Finally, a few concluding remarks are drawn in Section
\ref{sec:conclusion}. 
Table~\ref{tab:notation} provides a summary of the notation and symbols used throughout the paper.

\begin{table}[htbp]
    \centering
    \begin{tabular}{c|c||c|c}
    \toprule
    Symbol & Description & Symbol & Description\\
    \midrule
    $f_j$ & $f(x_j)$ & $\Delta x_i$ & $x_{i+1}-x_i$\\
    $\Delta f_i$ & $f_{i+1}-f_i$ & $m$ & window size\\
    ${m}_j$ & $\min \{m,j\}$ & $j_m$ & $\max \{0,j-m\}$ \\
    $\mathcal{X}_j$ & $[\Delta x_{j_m}\;\ldots\;\Delta x_{j-1}]$ & $\mathcal{F}_j$ & $[\Delta f_{j_m}\;\ldots\;\Delta f_{j-1}]$ \\
    \bottomrule
    \end{tabular}
    \caption{List of some notation and symbols used in this paper.}
    \label{tab:notation}
\end{table}

\section{Anderson Acceleration with Truncated Gram-Schmidt (AATGS)}
\label{sec:AATGS}
The variant of Anderson Acceleration to be introduced in this section relies on
building an orthonormal set of vectors which will be used in place of the set $\F_j$ in AA.
The idea of using an orthonormal basis in AA is not
completely new. For example, it is common to use the QR decomposition to
determine the minimizer $\theta_j$ in \eqref{eq:thetaj} by orthonormalizing the
columns of $\mathcal{F}_j$. This will lead to a process that is less prone to
numerical errors than an approach based on normal equations. However, in the
limited-depth case, this approach requires the successive QR
factorization of an evolving set of vectors in which the oldest vector is
removed at each step once the buffer that stores $\mathcal{F}_j$ is full - which
occurs when $j+1\geq m$. The proper way to implement this effectively
in order to obtain the QR factorization of each new set of vectors, is through
a simple QR-downdating scheme,  see, e.g., \cite{homer11}. 
In this paper, we adopt a different viewpoint, proceeding similarly to the truncated GCR algorithm \cite{Eis-Elm-Sch} to produce a `locally' orthonormal basis, i.e., a basis in which the last vector is orthogonal to the most recent $m_j-1$ vectors instead of all previous vectors. We will show that this variant has some advantages over classical AA.

\subsection{AATGS(m)}
The basic idea of AATGS($m$) is to exploit an evolving set of `locally
  orthonormal' vectors $\{ q_i \} $ to simplify and improve the solution of the
  least-squares problem \eqref{eq:thetaj}. At the $j$-th step we start with
  $m_j -1 $ such vectors $q_{j_m+1}, q_{j_m+2},\cdots, q_{j-1}$ (when $j=0$ this
  set is empty). We orthonormalize $\Delta f_{j-1}$ against these vectors to
  obtain the next member $q_j$ of the set. Now the set $\F_j$ in \eqref{eq:dfdx} is
  replaced by $Q_j = [q_{j_m+1}, \cdots, q_j]$ and so the least-squares
  problem \eqref{eq:thetaj} is trivial to solve: $\theta_j = Q_j^T f_j$ and 
  $\bar f_j = f_j - Q_j \theta_j$ replaces the $\bar f_j $ in \eqref{eq:AA1}. 
  However, the new intermediate solution $\bar x_{j}$  can no longer be written as
  $\bar x_j = x_j - \X_j \theta_j$ \emph{because the sets $\X_j, Q_j$ are no longer paired
    by a secant condition}. This can be remedied by replacing the set $\X_j$ by a set 
$U_j=[u_{j_m+1},u_{j_m+2},\ldots,u_j]$ which  is \emph{paired} with the $q_i$'s.
Here $u_j $ is initially set to $\Delta x_{j-1}$ and this is linearly combined with the previous
$u_i$'s in exactly the same way $\Delta f_{j-1}$ is combined to the previous $q_i$'s.
In this way  the sets  $Q_j$ and $U_j$ are paired by a secant relation, in the sense that
each $q_i$ is approximately $J u_i$ where $J$ is the Jacobian at $x_i$.
The relation $\bar f_j = f_j - Q_j \theta_j$ indicates that the correct $\bar x_j$ is 
$\bar x_j = x_j - U_j \theta_j$  and we can compute $x_{j+1} = \bar x_j + \beta_j \bar f_j$ as before.
The whole procedure  is sketched as Algorithm~\ref{alg:TGS}.

\begin{algorithm}[H]
  \centering
  \caption{AATGS(m)}\label{alg:TGS}
    \begin{algorithmic}[1]
  \State \textbf{Input}: Function $f(x)$, initial guess $x_0$, window size $m$ \\
  Set $f_0 \equiv f(x_0)$, \quad  $x_1=x_0+\beta_0f_0$, \quad 
  $f_1 \equiv f(x_1)$
\For{$j=1,2,\cdots,$ until convergence} 
\State $u := \Delta x = x_j-x_{j-1}$
\State $q := \Delta f = f_j - f_{j-1}$
\For{{$i=j_m+1,\ldots, j-1$}} 
\State $s_{ij} :=  (q, q_i) $
\State $u := u - s_{ij} u_i$
\State $q := q - s_{ij} q_i$
\EndFor
\State $s_{jj}=\Vert q \Vert_2$
\State $q_{j} := q /s_{jj}$, \quad  $u_{j} := u /s_{jj}$\
\State {Set ${Q}_j = [q_{j_m+1}, \ldots, q_{j}],\quad {U}_j = [u_{j_m+1}, \ldots, u_{j}]$ } 
\State Compute ${\theta}_j = {Q}^{\top}_jf_j$
\State $x_{j+1} = (x_{j}-{U}_j{\theta}_j)+\beta_{j} (f_j- {Q}_j{\theta}_j)$
\State $f_{j+1} = f(x_{j+1})$
\EndFor
\end{algorithmic}
\end{algorithm}

When $m=\infty$, Lines 6-10 in Algorithm~\ref{alg:TGS} perform a modified
  Gram-Schmidt process to orthonormalize $\Delta f_{j-1}$ against all previous
  $q_i$'s, resulting in the vector $q_j$.  When $m<\infty$, the same lines of
  pseudocode perform an incomplete orthonormalization via a \emph{truncated
    Gram-Schmidt} procedure in which $\Delta f_{j-1}$ is orthonormalized against the previous
  $m_j-1$ vectors $q_i$'s, resulting again in the vector $q_j$.
  In the loop the exact same linear transformation is applied to get $u_j$ from
  $\Delta x_{j-1}$ and the previous $u_i$'s.
  We prefer   the modified Gram-Schmidt (MGS) to the numerically unreliable
  Classical Gram-Schmidt (CGS), but we stress that
  CGS with reorthogonalization can also be useful in a
   parallel computing environment    \cite{Swirydowicz_Langou_Ananthan_Yang_Thomas_2020}.
   Throughout the entire AATGS iterations, the bases $Q_j$
  and $U_j$ will always contain at most $m$ vectors.  Figure \ref{fig:truncatedGS}
  shows an illustration of how the truncated Gram-Schmidt process operates for
  AATGS(3).
\begin{figure}[htb]
    \centering
    \includegraphics[width=0.3\linewidth]{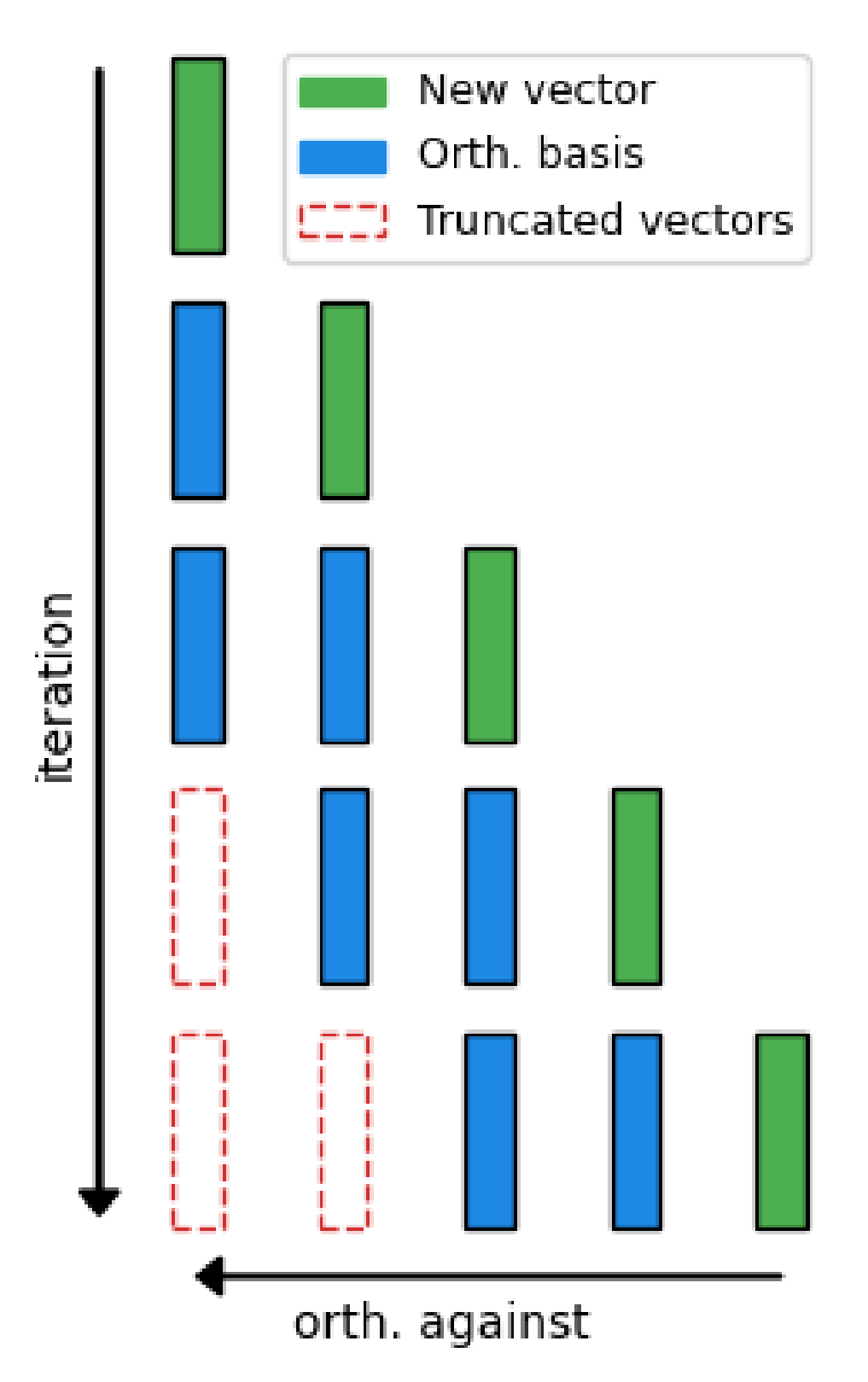}
    \caption{An illustration of the truncated Gram-Schmidt process to build the
      $q_i$'s in Lines 6-10 in Alg 2.1. In this figure,  the window size is $m=3$.
      The same  picture illustrates the process for the $u_i$'s: the new vector  $\Delta x_{j-1}$ is 
      linearly combined with (instead of orthonormalized against)  at most 2 previous $u_i$'s,  using the same scalars $s_{ij}$ as for the $q_i$'s.
    \label{fig:truncatedGS}}
\end{figure}

Define the
 $m_j \times m_j$ upper triangular matrix $S_j=\{s_{ik}\}_{i={j_m+1}:j,k={j_m+1}:j}$
resulting from the orthogonalization process, where the nonzero entries $s_{ij}$
are defined in Lines 7 and 11 of the algorithm.
In the full window case ($m = \infty$) or when $j \le m $ the block in Lines 4--12 essentially performs
a Gram-Schmidt QR factorization of the matrix $\F_j$, and enacts identical
operations on the set
$\X_j = [\Delta x_{j_m}, \Delta x_{j_m+1}, \cdots, \Delta x_{j-1} ]$.  A result
of the algorithm in this particular case is that we have
\eq{eq:QRrel} \F_j = Q_j S_j ; \qquad \X_j = U_j S_j,
\en 
but the above relation no longer holds when $j > m$.

Let us examine what happens  when $j=m+1$, focussing on  the set $Q_j$.  Before the orthogonalization begins we have
$\F_m = Q_m S_m$. To simplify notation we set $v_i \equiv \Delta f_{i-1}$ so that $\F_m=[v_1, v_2, \cdots, v_m] = Q_m S_m$.
Next the vector $q_{m+1}$ is computed and
  \emph{before truncation is applied} we actually have the factorization:
  $[v_1, v_2, \cdots, v_{m+1}] = [Q_m, q_{m+1} ]  S_{m+1}$ where $S_{m+1}$ is an
  $(m+1) \times (m+1) $ upper triangular matrix and $q_{m+1}$ is orthogonal to $q_2, \cdots, q_m$. This factorization is now truncated.  Ignoring the first column from the equality in this
relation leads to (recall that $Q_{m+1} \equiv [q_2, q_3, \cdots, q_{m+1}]$):
\[
  [v_2, \cdots, v_{m+1}] = [q_1, q_2, \cdots, q_{m+1} ]
  \begin{bmatrix}s_1^T \\ \hat S_{m+1} 
  \end{bmatrix} = q_1 s_1^T + Q_{m+1} \hat S_{m+1} , 
\]
where,  using matlab notation: $s_1^T \equiv S_{m+1} (1,2:m+1)$, an $1 \times m$ vector,  and
$\hat S_{m+1} \equiv S_{m+1} (2:m+1,2:m+1)$, an $m \times m$ matrix.  
The end result is that the pair of matrices $Q_{m+1}$ and $\hat S_{m+1}$ constitute
the QR factorization of $[v_2, \cdots, v_{m+1}] - q_1 s_1^T$. Note that $q_1$ is a mutiple of $v_1$.
In the next step the set is modified by a rank-one matrix of the form $q_2 s^T$ and the process
continues in the same way adding one more rank-one perturbation at each step.

After step $j$ of Algorithm~\ref{alg:TGS} is applied, we would have built
the orthonormal basis $Q_j = [q_{j_m+1}, \cdots, q_j] $ along with a paired
system $U_j = [u_{j_m+1}, \cdots, u_j]$.  Note that $Q_j$ has orthonormal
columns but not $U_j$.
The vector $\theta_j$ computed in Line~14 by a simple matrix-vector product,
is the least-squares solution of $ \min_\theta \| f_j - Q_j \theta \|_2$.
The resulting residual $f_j - Q_j \theta_j$ is not 
not necessarily the same as the  $\bar f_j = f_j - \X_j \theta_j$ of AA 
since the span of $\F_j$ differs from the span
of $Q_j$ when $j > m$. 
 Line 15 computes the next iterate $x_{j+1}$ using
the two paired bases $Q_j$ and $U_j$ and ${\theta}_j$.
Note that as for classical Anderson, we can also define
$\bar{x}_j$ and $\bar{f}_j$ and rewrite $x_{j+1}$ in the following form: 
\eq{eq:AA2N} 
\bar x_j = x_j-U_j {\theta}_{j} ,\quad 
\bar f_j = f_j-Q_j  {\theta}_{j} , \quad 
x_{j+1} =\bar x_j + \beta_j \bar f_j .
\en

We mentioned the case $j>m$ in the above discussion. It is easy to see that 
in the case where $j\le m$, the subspaces
spanned by $\F_j$ and $Q_j$ are identical and in this situation the iterates
$x_{j+1}$
resulting from AA and AATGS will be the same. In  particular,
when $m = \infty$ this will always be the case, i.e., 
the full-depth AATGS($\infty$) and AA$(\infty)$ return the same iterate $x_{j+1}$ in exact arithmetic at each iteration and thus are mathematically equivalent.

In the next section, we will study the properties of AATGS($\infty$) and exhibit a
particularly interesting short-term recurrence of the algorithm when it is applied to
symmetric linear systems.
\subsection{Theoretical analysis of AATGS($\infty$)}
Consider a linear problem where $f(x)=b-Ax$ and $A$ is invertible. Note that in
this case we have \eq{eq:APj} \F_j = -A \X_j.  \en In the next lemma, we show
that the matrix $U_j$ returned by Algorithm \ref{alg:TGS} forms a basis of the Krylov
subspace $\mathcal{K}_j(A,f_0)$ and that under mild conditions,
$Q_j, U_j$ satisfy the same relation as $\F_j, \X_j$ in \eqref{eq:APj}  for AATGS($\infty)$.

\begin{lemma}
  Assume $A$ is invertible and $f(x)=b-Ax$. If Algorithm \ref{alg:TGS}
 applied for solving $f(x) = 0 $  with
  $m=\infty$ does not break at step $j$, then the system
  $U_j$ forms a basis of the Krylov subspace $\mathcal{K}_j(A,f_0)$. In
  addition, the orthonormal system $Q_j$ built by Algorithm \ref{alg:TGS}
  satisfies $Q_j= -AU_j$.
\label{lemma:QAU}
\end{lemma}
\begin{proof}
  We first prove $Q_j=-AU_j$ by induction. When $j=1$, we have
  $q_1=(f_1-f_0)/s_{11}=-Au_1$. Assume $Q_{j-1}=-AU_{j-1}$. Then we have
    \begin{align*}
        s_{jj}q_j &= (f_j-f_{j-1})-\sum_{i=1}^{j-1}s_{ij}q_{i} =-A(x_j-x_{j-1})-\sum_{i=1}^{j-1}s_{ij}(-Au_{i})\\
        & = -A[(x_j-x_{j-1})-\sum_{i=1}^{j-1}s_{ij}u_{i}]\\ 
        & = s_{jj}(-Au_j).
    \end{align*}
    Thus, since $s_{jj} \ne 0$ we get $q_j=-Au_j$ and therefore $Q_j=-AU_j$, completing the
    induction proof.
    
    Next, we prove by induction  that $U_j$ forms a basis of $\mathcal{K}_{j}(A,f_0)$.
    It is more convenient to prove by induction the property that 
    for each $i\le j$, $U_i$ forms a basis of $\mathcal{K}_{i}(A,f_0)$.
    The result is true for $j=1$ since we have
    $u_1 = (x_1 - x_0)/s_{11} = \beta_0 f_0 /s_{11}$.  
    Now let us assume the property is true for $j-1$, i.e., that for each $i=1,2,\cdots,j-1$,
    $U_{i}$ is a basis of the Krylov subspace $\mathcal{K}_{i}(A,f_0)$. Then we have
\begin{align}
  s_{jj}u_j
  &= (x_j-x_{j-1})-\sum_{i=1}^{j-1}s_{ij}u_i \label{eq:sjjuj} \\                  
  &= -U_{j-1}{\theta}_{j-1} +\beta_{j-1}(f_{j-1}-Q_{j-1}{\theta}_{j-1})-\sum_{i=1}^{j-1}s_{ij}u_i \nonumber\\
  & = -U_{j-1}{\theta}_{j-1} + \beta_{j-1}f_{j-1} -\beta_{j-1}Q_{j-1}{\theta}_{j-1}- \sum_{i=1}^{j-1}s_{ij}u_i\nonumber \\
  & =\beta_{j-1}f_{j-1} -U_{j-1}{\theta}_{j-1}+\beta_{j-1}AU_{j-1}{\theta}_{j-1}-\sum_{i=1}^{j-1}s_{ij}u_i. \nonumber
\end{align}
  The induction hypothesis  shows that 
  $-U_{j-1}{\theta}_{j-1}+\beta_{j-1}AU_{j-1}\theta_{j-1}-\sum_{i=1}^{j-1}s_{ij}u_i\in\mathcal{K}_{j}(A,f_0)$.
  It remains to show that  $f_{j-1} = b-Ax_{j-1}\in \mathcal{K}_{j}(A,f_0)$. For this, we expand 
  $b-Ax_{j-1}$ as
    \[
        b-Ax_{j-1} = b- Ax_{j-1}+Ax_{j-2}-Ax_{j-2}+\ldots-Ax_{1}+Ax_{0}-Ax_{0} = \sum_{i=1}^{j-1}-A(x_{i}-x_{i-1})+f_0.
    \]
    From the relation \eqref{eq:sjjuj} applied with $j$ replaced by $i$, we see that $x_i - x_{i-1}$ is a linear
    combination of $u_1, u_2, \cdots, u_i$, i.e., it is a member $\K_i$ by the induction hypothesis.
Therefore $ -A (x_i - x_{i-1}) \in \K_{i+1}$ - but since $i \le j-1$ then
$ - A (x_i - x_{i-1}) \in \K_{j}$.  The remaining term $f_0$ is clearly in $\K_j$. Because $U_{j}=-A^{-1}Q_j$ has full column rank and $u_i\in\mathcal{K}_{j}(A,f_0)$ for $i=1,\ldots,j$, $U_j$ forms a basis of $\mathcal{K}_{j}(A,f_0)$. This completes the induction proof. \end{proof}

From \eqref{eq:AA2N}, we see that in the linear case under consideration
the vector $\bar f_j$ is the residual for $\bar x_j$:
\eq{eq:fbarj}
\bar f_j = f_j- Q_j {\theta}_{j}= (b-A x_j) - Q_j {\theta}_{j} = (b-A x_j) + A U_j {\theta}_{j} =
b - A (x_j - U_j {\theta}_{j})
= b - A \bar x_j .
\en   
The next theorem shows that $\bar x_j$ minimizes $\| b - A x\|_2$ over the affine space $x_0+\mathcal{K}_j{(A,f_0)}$.

\begin{theorem}
  \label{lem:opt1}
  The vector  $\bar x_{j}$ generated at the $j$-th step
  of AATGS($\infty$) minimizes the residual norm $\| b - A x\|_2$
  over all vectors $x$ in the affine space $x_0 + \mathcal{K}_j(A,f_0)$.
  It also minimizes the same residual norm over the subspace $x_k + \mathcal{K}_j(A,f_0)$ for any $k$ such that $0 \le k \le j$.
  \end{theorem}

  \begin{proof}
    Consider a  vector of the form $x = x_j - \delta $ where $\delta = U_j y$
    is an arbitrary member of $\mathcal{K}_j(A,f_0)$. We have
    \eq{eq:prf1}
      b - A x  = b - A (x_j - U_j y) = f_j + A U_j y = f_j - Q_j y .
      \en
      The minimal norm $\|b-Ax\|$ is reached when $y = Q_j^{\top} f_j $
      and the corresponding optimal $x$ is $\bar x_j$. Therefore, $\bar x_j$
      is the vector $x $ of the affine space $x_j + \mathcal{K}_j(A,f_0)$  with the smallest
      residual norm. We now write $x$ as:
\begin{align} 
x &= x_j - U_j y \nonumber   \\
  &= x_0 + (x_1-x_0) + (x_2-x_1) +
    (x_3-x_2) + \cdots (x_{i+1} - x_i) +\cdots (x_j - x_{j-1}) - U_j y \label{eq:lem2.2prf1}\\
  &= x_0 + \Delta x_0 + \Delta x_1 + \cdots +\Delta x_{j-1} - U_j y.  \label{eq:lem2.2prf2}
\end{align} 
 We will exploit the relation obtained from the QR factorization of Algorithm \ref{alg:TGS}, namely $\X_j = U_j S_j $ in \eqref{eq:QRrel}:
If $e$ is the vector of all ones, then
$\Delta x_0 + \Delta x_1 + \cdots + \Delta x_{j-1}  = \X_j e = U_j S_j e$.
Define $t_j \equiv S_j e$. Then, from \eqref{eq:lem2.2prf2} we obtain 
\eq{eq:xepr00}
x = x_j - \delta  = x_0 - U_j [y - t_j] .
\en 
This means that the set of all vectors of the form $ x_j - \delta $ is the same as
the set of all vectors of the form $x_0 - \delta'$ where $\delta' \ \in \ \mathcal{K}_j(A,f_0)$.
As a result, $\bar x_j$ also minimizes $b - A x$ over all vectors in the affine
space
$x_0 +  \ \mathcal{K}_j(A,f_0)$.
The proof can be easily repeated for any $k$ between $0$ and $j$. The expansion (\ref{eq:lem2.2prf1} --\ref{eq:lem2.2prf2}) 
becomes 
\begin{align} 
 x_j - U_j y &= x_k + (x_{k+1}-x_k) + (x_{k+2}-x_{k+1}) + \cdots (x_{i+1} - x_i) +\cdots (x_j - x_{j-1}) - U_j y \label{eq:lem2.2prf3}\\
  &= x_k + \Delta x_k + \Delta x_{k+1} + \cdots + \Delta x_{j-1} - U_j y. 
\label{eq:lem2.2prf4}\
\end{align} 
The rest of the proof is similar and straightforward.
\end{proof}

Theorem \ref{lem:opt1} shows that $\bar x_j$ is the $j$-th iterate of the GMRES algorithm for solving
$Ax=b$ with the initial guess $x_0$ and that $\bar f_j $ is the corresponding residual. The value of $\bar x_j$ is independent of the choice of $\beta_i$ for $i\leq j$. Now consider the residual $f_{j+1}$ of AATGS($\infty$) at step $j+1$. From the relations
$x_{j+1} = \bar x_j + \beta_j \bar f_j$ and \eqref{eq:fbarj} we get:
\eq{eq:fjp1}
f_{j+1} = b - A[\bar x_j + \beta_j \bar f_j]  =b -  A \bar x_j - \beta_j A \bar f_j
 = \bar f_j - \beta_{j} A \bar f_j = (I-\beta_j A) \bar f_j.
\en
This implies that the vector $f_{j+1}$ is the residual for $x_{j+1}$ obtained from $x_{j+1} = \bar x_j + \beta_j \bar f_j$ - which is a simple Richardson iteration
starting from the iterate $\bar x_j$. Therefore,  $x_{j+1}$ 
in Line~15 of Algorithm \ref{alg:TGS}  
is nothing but a Richardson iteration step
from this GMRES iterate. This is stated in the following proposition.

\begin{proposition}
  \label{prop:opt1}
  Assume $A$ is invertible and $f(x)=b-Ax$. If Algorithm \ref{alg:TGS}
 applied for solving $f(x) = 0 $  with
  $m=\infty$ does not break at step $j+1$, then
  the residual $f_{j+1}$  of the iterate   $x_{j+1}$ generated at the $j$-th step
  of AATGS($\infty$) is equal to  $(I -\beta_j A) \bar f_j$ where $\bar f_j = b - A \bar x_j$ 
  minimizes the residual norm $\| b - A x\|_2$  over all vectors $x$ in the affine space $x_0 + \mathcal{K}_j(A,f_0)$.
  In other words, the $(j+1)$-st iterate of AATGS($\infty$) can be obtained by performing one step of a Richardson iteration applied to the
  $j$-th GMRES iterate.
\end{proposition}

A similar result has also been proved for the standard AA by Walker and Ni \cite{homer11} under slightly different assumptions.

\subsection{Short-term recurrence in AATGS for linear symmetric problems}\label{sec:st_aatgs}
We now show that the orthogonalization process (Lines 6-10 of Algorithm
\ref{alg:TGS}) simplifies in the linear symmetric case under consideration.
Indeed, we will see that $S_j$ consists of only 3 non-zero diagonals in the
upper triangular part when $A$ is symmetric. This implies that we only need to save
$q_{j-2}, q_{j-1}$ and $u_{j-2}, u_{j-1}$ in order to generate $q_j$ and $u_j$
in the full-depth AATGS($\infty$). Before we
prove this result, we first examine the 
components of the vector $Q^{\top}_jf_j$ in Line 14
of Algorithm \ref{alg:TGS}.

\begin{lemma}
  When $f(x)=b-Ax$ where $A$ is a real non-singular symmetric matrix
  then the entries of the vector ${\theta}_j=Q^T_jf_{j}$ in Algorithm \ref{alg:TGS} are all zeros except the last two.
    \label{lemma:twoentries}
\end{lemma}
\begin{proof}
 Let $i\leq j-1$. From \eqref{eq:fjp1}, we have 
    \[
    (f_{j},q_i) = (\bar{f}_{j-1}-\beta_{j-1}A\bar{f}_{j-1},q_i) = (\bar{f}_{j-1},q_i)-\beta_{j-1}(A\bar{f}_{j-1},q_i).
    \]
    The first term equals zero because $(f_{j-1}-Q_{j-1}{\theta}_{j-1},q_i)=((I-Q_{j-1}Q_{j-1}^T)f_{j-1},q_i)=0$. Consider the second term:
    \[
    (A\bar{f}_{j-1},q_i)=(\bar{f}_{j-1},Aq_i).
    \]

Observe that since $u_i\in\mathcal{K}_{i}(A,f_0)$,
then $q_i=-Au_i$ belongs to the Krylov subspace
$\mathcal{K}_{i+1}(A,f_0)$ which is the same as $\text{Span}  \{ U_{i+1} \}$
according to Lemma \ref{lemma:QAU}. Thus,
it can be written as $ Au_i = U_{i+1} y$ for some $y$ and
hence, $A q_i = -A U_{i+1} y = Q_{i+1} y$, i.e., $Aq_i$ is in the span
of $q_1, \cdots, q_{i+1} $.
    Therefore, recalling that $\bar{f}_{j-1} \perp \text{Span}  \{Q_{j-1}\}$, we have: 
    \[
    (\bar{f}_{j-1},Aq_i)=0 \quad \text{for} \quad i\leq j-2.
    \]
    In the end, we obtain $(f_{j},q_i)=0$ for $i\leq j-2$.
\end{proof}
Lemma \ref{lemma:twoentries} indicates that the computation of $x_{j+1}$ in Line
15 of Algorithm \ref{alg:TGS} only depends on the two most recent $q_{i}$'s and
$u_{i}$'s. The next theorem will further show that $q_j$ and $u_j$ in Line 12 can be
computed based on $q_{j-2},q_{j-1}$ and $u_{j-2},u_{j-1}$ instead of all
previous $q_i$'s and $u_i$'s.
\begin{theorem}
    When $f(x)=b-Ax$ where $A$ is a real non-singular symmetric matrix, then the upper triangular matrix $S_k$ is banded with bandwidth $3$, i.e., we have $s_{ij}=0$ for $i<j-2$. 
    \label{thm:thetaj}
\end{theorem}
\begin{proof}
    It is notationally more convenient to consider column $j+1$ of $S_k$ where $k>j$.   Denote $\Delta f_j=f_{j+1}-f_{j}$, and $\Delta x_j=x_{j+1}-x_{j}$. Consider $s_{i,j+1}=(\Delta f_{j},q_i)$ for $i\leq j$ and note that $s_{i,j+1} = -(A\Delta x_{j},q_i)$. We note that
    \[
        \Delta x_{j} = x_{j+1} - x_{j} = \bar{x}_j+\beta_j\bar{f}_j-x_j = x_j-U_j{\theta}_{j} + \beta_j\bar{f}_j-x_j = -U_j{\theta}_{j} + \beta_j\bar{f}_j.
    \]
    We write 
    \begin{align*}
        A\Delta x_{j} &= -AU_j{\theta}_{j}+\beta_jA\bar{f}_j = Q_j{\theta}_{j}+\beta_jA\bar{f}_j\\
        &= -(f_j-Q_j{\theta}_{j})+f_j+\beta_jA\bar{f}_j\\
        &= -\bar{f}_j + f_j+\beta_jA\bar{f}_j,
    \end{align*}
and hence, 
\begin{equation}
    (A\Delta x_{j},q_i) = -(\bar{f}_j,q_i)+(f_j,q_i)+\beta_j(A\bar{f}_j,q_i).
\end{equation}
The first term on the right-hand side, $(\bar{f}_j, q_i)$ vanishes since $i\leq j$. According to Lemma \ref{lemma:twoentries} the inner product $(f_j, q_i)$ is zero for $i\leq j-2$. The last term $(A\bar{f}_j,q_i)$ is equal to zero when $i\leq j-1$ as shown in the proof of Lemma \ref{lemma:twoentries}.  This completes the proof as it shows that $s_{i,j+1}=0$ for $i<j-1$.
\end{proof}

Lemma \ref{lemma:twoentries} and Theorem \ref{thm:thetaj} show that when
AATGS($\infty$) is applied to solving linear symmetric problems, only the two most
recent $q_{j-2}, q_{j-1}$ and $u_{j-2}, u_{j-1}$ are needed to compute the next
iterate $x_{j+1}$, which significantly reduces both memory and orthogonalization
costs. In other words, AATGS(3) is equivalent to AATGS($\infty$) in the linear symmetric
case.

\begin{corollary}
    When $f(x) = b - Ax$ where $A$ is non-singular and real symmetric,
    AATGS(3) is equivalent to AATGS($\infty$). 
\end{corollary}

Staying with the linear case, the next theorem examines the convergence rate of
AATGS($\infty$) when $A$ is symmetric positive definite.

\begin{theorem}
\label{thm:convegencespd}
Assume that $A$ is symmetric positive definite and that a constant $\beta$ is used in
AATGS. Then we have the following error bound for the residual  
$r^{AATGS}_{j+1}$ obtained  at the $(j+1)$-st step of AATGS($\infty$): 
\begin{equation}
       \Vert r^{AATGS}_{j+1}\Vert_2
       \leq 2\Vert I-\beta A
       \Vert_2\left(\frac{\sqrt{\kappa(A)}-1}{\sqrt{\kappa(A)}+1}\right)^j \Vert r_0\Vert_2,
\end{equation}
where $\kappa(A)$ is the 2-norm condition number of $A$.
\end{theorem}

\begin{proof} 
Based on Proposition 2.3, we have
\begin{equation*}
\Vert r^{AATGS}_{j+1}\Vert_2 = \Vert (I-\beta A) r_{j}^{GMRES}\Vert_2,
\end{equation*}
where $r_{j}^{GMRES}$ denotes the residual associated with the $j$-th iterate from GMRES.
Since $A\in\mathbb{R}^{n\times n}$ is symmetric, it admits the following eigendecomposition:
\begin{equation}
    A = {U}{\Lambda}{U}^{\top}, \quad {U^\top} {U} = I, \quad \Lambda = \operatorname{diag}(\lambda_1,\dots,\lambda_{n}),
\end{equation}
where $ 0 < \lambda_1 \le \lambda_2 \le \cdots \le \lambda_n $.
It is known that the GMRES residual vector can be expressed as
\begin{equation}
    r^{GMRES}_{j} = \rho (A)r_0={U}\rho ({\Lambda}){U}^{\top}r_{0}, ~~~ \rho \in \mathcal{P}_{j},
\end{equation}
where $\mathcal{P}_{j}$ is the affine space of polynomials  $p$ of degree $j$ such that
$p(0) = 1$  and 
 \begin{align}
   \|\rho (A)r_0\|_2
   &=\min_{p\in\mathcal{P}_j} \Vert p(A)r_0\Vert_2\leq
     \min_{p\in \mathcal{P}_{j}}\max_{i}| p(\lambda_i)|\Vert r_0\Vert_2 \\
   & \le\min_{p\in \mathcal{P}_{j}} \ \ \max_{\lambda \in [\lambda_1, \lambda_n]}| p(\lambda)|\Vert r_0\Vert_2 \\
   &   \le \frac{\Vert r_0\Vert_2}{T_{j}(1 + 2\frac{\lambda_1}{\lambda_n - \lambda_1})},
\end{align}
where $T_{j}$ is the Chebyshev polynomial of first kind of degree $j$.
The last inequality follows from well-known results on the optimality properties of Chebyshev polynomials,
see, e.g., \cite{Saad-book2}.
Also note that since
\begin{equation*}
    T_j(\lambda)\geq \frac12\left(\lambda+\sqrt{\lambda^2-1}\right)^j,
\end{equation*}
we have 
\begin{equation*}
T_j\left(1 + 2\frac{\lambda_1}{\lambda_n - \lambda_1}\right)\geq \frac12\left(\frac{\sqrt{\kappa(A)}+1}{\sqrt{\kappa(A)}-1}\right)^j.
\end{equation*}

Thus, we obtain  
\begin{align*}
\Vert r^{AATGS}_{j+1}\Vert_2 &= 
  \Vert (I-\beta A) r_{j}^{GMRES}\Vert_2\\ &= \Vert (I-\beta A)\rho (A)r_0\Vert_2 \\
    & \leq \Vert (I-\beta A)\Vert_2 \Vert \rho (A) r_{0}\Vert_2\\
& \leq 
\frac{\Vert (I-\beta A)\Vert_2}{T_{j}(1 + 2\frac{\lambda_1}{\lambda_n - \lambda_1})}\Vert r_0\Vert_2 \\
& \leq  
2\Vert I-\beta A
\Vert_2\left(\frac{\sqrt{\kappa(A)}-1}{\sqrt{\kappa(A)}+1}\right)^j \Vert r_0\Vert_2.
\end{align*}
This completes the proof. 
\end{proof}

The convergence results can be generalized to the case where the eigenvalues of $A$ are distributed in two intervals excluding the origin. This result is omitted. 

Another case of interest is when $A$ is skew-symmetric.
In this situation, when  the $\beta_j$'s are constant, it can be seen that the AATGS algorithm  yields
$x_2= x_1$ after the first iteration, and consequently, the process breaks at Line 12 due to $s_{22}$
being equal to zero. To circumvent this problem, one could adjust 
$\beta_j$ at each iteration. Alternatively, reformulating the problem $f(x)$
itself presents another viable strategy. An example demonstrating this approach
is provided in Section \ref{sec:minimax} for solving minimax optimization
problems. Note that there is no issue in the interesting case when $A$ is of the form
$A = I + S$ where $S$ is skew-symmetric, for which it can be shown that we do have a simplification similar to that
of the symmetric case. 

 \subsection{Limited-depth AATGS} 
We now explore the limited-depth version of AATGS($m$) for a fixed $m$.   Recall the notation $j_m = \max \{0,j-m\}$.   
At step $j$, Algorithm \ref{alg:TGS} orthogonalizes the latest  $\Delta f$ vector against
$q_{j_m+1}, ..., q_{j-1}$ to produce $q_j$. We set ${U}_j \equiv [u_{j_{m}+1}, u_{j_m+2}, \cdots, u_{j-1},u_j]$ and 
${Q}_j \equiv [q_{j_{m}+1}, q_{j_m+2}, \cdots, q_{j-1},q_j]$ in Line 13. Note that $U_j$ and $Q_j$ have
$\min \{j, m\}$ columns, which is the same number of  columns as the block ${\cal F}_j$ in 
Anderson Acceleration. 
As it turns out, $\bar x_j=x_{j}-{U}_j{\theta}_j$ satisfies a similar result to that of
Theorem~\ref{lem:opt1}.

\begin{proposition}\label{lem:opt2} The intermediate iterate 
  $\bar x_j=x_{j}-{U}_j{\theta}_j$ obtained at the $j$-th step of AATGS(m) minimizes $\| b  - Ax \|_2$ over all vectors $x$ of the form
$x = x_{j} - \delta $ where $\delta \in \text{Span} \{{U}_j \}$.
\end{proposition}
 
\begin{proof}
We consider a generic vector $x =  x_j - \delta $ where
$\delta \ \in \ \text{span} \{{U}_{j} \} $ which we write as $\delta = {U}_j y$.
Then Equation~\eqref{eq:prf1} in the proof of Theorem~\ref{lem:opt1}
still holds, i.e., we can write
$ r\equiv b - A x  = f_j - {Q}_j y .$
It is known that the residual norm is minimal iff
$r\perp \text{Span} \{ {Q}_j \}$, i.e., iff:
$f_j - {Q}_j y \perp \text{Span} \{ {Q}_j \}$ which is precisely the condition
imposed to get ${\theta}_j $.
This  means that $\bar x_j$   minimizes $\| b  - Ax \|_2$ over all
  vectors $x$ of the form
  $x = x_{j} - \delta $ where $\delta \in \text{Span} \{{U}_j \}$.
\end{proof} 
Note that this result is a little weaker than that of Theorem~\ref{lem:opt1} which allowed the
affine spaces on which the residual norm is minimized to be of the form
$x_k + \text{Span} \{ U_j \}$ for any $k$ between 0 and $j$.  
Similar to the full-depth case, 
we may now ask whether the vector $\bar x_j$ corresponds to
the result of some other classical algorithms for linear systems.
One may think that there should exist an equivalence with a similar method such as Truncated GCR
(TGCR also known as ORTHOMIN, see e.g.,  \cite{Saad-book2}) or one of the other
Krylov methods that rely on truncation in the orthogonalization, e.g., 
ORTHODIR, or DQGMRES \cite{Saad-book2}. While this is possible, we did not find an obvious
result that showed such an equivalence.

\section{Restarting AATGS}
\label{sec:restarting}
Lines 6-12 of Algorithm~\ref{alg:TGS} carry-out an orthonormalization of the vector
$q_j$ versus $q_{j_m+1}, \cdots, q_{j-1}$ and imposes the same operations undergone by the sequence
$\{ q_i \}$ to the sequence $\{u_i\}$. While the columns of 
$Q_j$ are orthonormal, those of $U_j$ are not, and they are prone to numerical instability.
Therefore, it is essential to check for the onset of instability, especially when the 
problem is neither linear nor positive definite. To take advantage of the
short-term recurrence while also preserving accuracy, we introduce a lightweight
strategy to determine when a restart is deemed necessary. 

Using the same notation as in Algorithm \ref{alg:TGS}, the propagation of $U_j$
can be expressed in the following matrix form:
\begin{equation}
    \begin{bmatrix}
        u_{j_m+2}^T \\
        \vdots \\
        u_{j-1}^T \\
        u_{j}^T
    \end{bmatrix}
    =
    \begin{bmatrix}
        0 & 1 \\
        \vdots & & \ddots \\
        0 & & & 1 \\
        -\frac{s_{j_m+1,j}}{s_{jj}} & -\frac{s_{j_m+2,j}}{s_{jj}} & \cdots & -\frac{s_{j-1,j}}{s_{jj}} \\
    \end{bmatrix}
    \begin{bmatrix}
        u_{j_m+1}^T \\
        \vdots \\
        u_{j-2}^T \\
        u_{j-1}^T
    \end{bmatrix}
    +
    \frac{1}{s_{jj}}
    \begin{bmatrix}
        0 \\
        \vdots \\
        0 \\
        \Delta x^T
    \end{bmatrix}
\label{eq:envolve_u}
\end{equation}
where $\Delta x := x_j - x_{j-1}$.  Note that the above system need not be formed 
explicitly. We point out 
that \eqref{eq:envolve_u} is applied element-wise to the columns of $U_j$. This
means that the $k$-th component of $u_j$ can be derived by applying the operations
in \eqref{eq:envolve_u} to the $k$-th element of $\Delta x$ and the $k$-th row
of $U_j$, i.e., if $v\up{k}$ refers to the $k$-th component of a vector $v$, we have
\begin{equation}
    u_j^{(k)} = \frac{1}{s_{jj}}\Delta x^{(k)} - \sum_{i=j_m+1}^{j-1} \frac{s_{ij}}{s_{jj}} u_i^{(k)}.
\label{eq:ee1}
\end{equation} 
We now analyze how the accumulation of the errors from the computation of previous $u_i$'s 
affect the accuracy of the most recent $u_j$. For this we
denote the computed version of $u_i$ as $\tilde{u}_i = u_i + \varepsilon_i$, where $\varepsilon_i \in \mathbb{R}^n$ represents the error introduced during the computation of $u_i$.
We also denote  the rounding errors introduced during the computation of $u_j$ at step $j-1$ by
$\delta_j$ and we assume that
\begin{equation}
\label{eq:assumption}
\|\delta_j\|_\infty \leq C \cdot \|\Delta x\|_\infty/s_{jj} 
\end{equation}
where $C$ is a constant.

Then, the perturbed version of \eqref{eq:ee1} becomes:
\begin{equation}
    \tilde{u}_j^{(k)} = \frac{1}{s_{jj}}
    \Delta x^{(k)} - \sum_{i=j_m+1}^{j-1} \frac{s_{ij}}{s_{jj}} \tilde{u}_i^{(k)} + \delta_j^{(k)}.
\label{eq:ee2}
\end{equation}
We then substitute $\tilde{u}_j = u_j + \varepsilon_j$ and $\tilde{u}_i = u_i + \varepsilon_i$ into Equation \eqref{eq:ee2} and subtract Equation \eqref{eq:ee1}. This  leads to:
\begin{equation}
    \varepsilon_j^{(k)} = - \sum_{i=j_m+1}^{j-1} \frac{s_{ij}}{s_{jj}} \varepsilon_i^{(k)} + \delta_j^{(k)}.
\end{equation}
Therefore,
\begin{align}
  |\varepsilon_j^{(k)}| &\leq \sum_{i=j_m+1}^{j-1} \frac{|s_{ij}|}{s_{jj}} |\varepsilon_i^{(k)} | + |\delta_j^{(k)}|
  \nonumber \\
  &\leq \sum_{i=j_m+1}^{j-1} \frac{|s_{ij}|}{s_{jj}} \|\varepsilon_i\|_\infty + \frac{C}{s_{jj}} \|\Delta x\|_\infty .
  \label{eq:ineqstab0}                         
\end{align}
This leads us to define the following scalar sequence $w_j$ to monitor the behavior of the bound
\eqref{eq:ineqstab0}:
\begin{align}    w_j := \sum_{i=j_m+1}^{j-1} \frac{|s_{ij}|}{s_{jj}} w_i + \frac{C}{s_{jj}}\|\Delta x\|_\infty.
\label{eq:envolve_w}
\end{align} 
The sequence $w_j$ is just an upper bound for the infinity norm of the error vector $\varepsilon_j $ and it can be used to monitor the growth of the rounding errors. When $w_j$ exceeds a threshold $\eta>0$, we should discard all vectors in $U_j$ and $Q_j$ and {set $j_m \equiv j$}. The next iteration then computes $\Delta x$ and $\Delta f$ in Lines~4-5 of Algorithm~\ref{alg:TGS} using the latest pairs $x_{j+1}, x_j$ along with related $f_{j+1}, f_j$ and set $w_{j+1}= C\cdot\|\Delta x\|_\infty / s_{{j+1},{j+1}}$ to restart monitoring the growth of the rounding errors.
The auto-restart version of Algorithm \ref{alg:TGS} is briefly summarized in Algorithm
\ref{alg:restart}. In the  experiments section, we set
$C = 1$, unless otherwise specified.

\begin{algorithm}[H]
\centering
\caption{AATGS(m) with Restarting}\label{alg:restart}
\begin{algorithmic}[1]
    \State \textbf{Input}: Function $f(x)$, initial guess $x_0$, window size $m$, threshold $\eta$, constant $C$.
    \State Set $f_0 \equiv f(x_0)$, \quad  $x_1=x_0+\beta_0f_0$, \quad $f_1 \equiv f(x_1)$, \quad $w_0 \equiv 0$, \quad $j_m=0$
    \For{$j=1,2,\cdots,$ until convergence} 
               \State Update $j_m := \max\{j_m, j-m\}$ 
        \State Run lines 4-16 of Algorithm \ref{alg:TGS}
            \State $w_j := C\cdot\|\Delta x\|_\infty / s_{jj} + \sum_{i=j_m+1}^{j-1} (|s_{ij}|/s_{jj})w_i$
        \If{$w_j > \eta$}
            \State {Set $j_m \equiv j$} and $Q_j \equiv [~], \ U_j \equiv [~]$
            \EndIf
    \EndFor
\end{algorithmic}
\end{algorithm}

Here are some details and comments on Algorithm \ref{alg:restart}:
\begin{itemize} 
    \item  \textbf{Line 2:} 
      Same as the initialization step in Algorithm \ref{alg:TGS}.
      In the implementation, we can allocate a vector of length $m$ to store $w_j$'s.

    \item \textbf{Line 6:}
    Note that when $j-1<j_m$, i.e., in the first step after a restart, the sum in the expression is empty and therefore equal to zero. In this case both $Q_j$ and $U_j$ are empty. In this situation   
    $w_j := C \|\Delta x\|_\infty / s_{jj}$ reflecting the fact that there are no errors propagating from earlier steps.  
    
    \item \textbf{Lines 7-9:} When $w_j$ surpasses the given threshold $\eta$,
    a restart is necessary because the stability is compromised. For a restart,
    we set $j_m \equiv j$ and discard all stored vectors in $Q_j$ and
    $U_j$. We only retain the last two iterates, $x_j$ and $x_{j+1}$ as well as $f_j$ and $f_{j+1}$ to continue the process
    when we compute $\Delta x$ and $\Delta f$ in the next iteration. 
    Algorithm \ref{alg:restart} will generate mathematically the same iterates as Algorithm \ref{alg:TGS} if this condition is not met.
\end{itemize}

\section{Experiments} \label{sec:exp}
This section presents a few experiments on nonlinear problems to compare AATGS with the standard AA.
We also include the results of the fixed point iteration in our experiments.
Since it is common practice to add a fixed restart for AA (i.e., clearing $\mathcal{X}_j$ and $\mathcal{F}_j$ every fixed number of iterations), we incorporate a fixed restart for both AATGS (in addition to the auto-restart strategy discussed in Section~\ref{sec:restarting}) and AA.
For AA, to restart after obtaining $x_{j+1}$, we discard all vectors in $\mathcal{X}_j$ and $\mathcal{F}_j$.
The next iteration then begins by computing $\Delta x_j$ and $\Delta f_j$ using $x_{j+1}$ and $x_j$ along with related $f_{j+1}$ and $f_j$.
In the figures in this section, we use the notation AATGS[$m,d$] and AA[$m,d$] to represent AATGS and AA with window size $m$ and fixed restart dimension $d$.
When $d$ is replaced with a `$-$', fixed restart is disabled.
Unless otherwise noted, we set the threshold parameter $\eta$ in Algorithm \ref{alg:restart} to $10^3$.
Note that standard AA performance is highly dependent on window size $m$ and fixed restart dimension $d$. 
While we present results for only a few AA parameters, we employ a grid search to select the best-performing AA configurations.
In our tests, the max number of iterations and stopping tolerance for the relative norm of $f(x)$ varies based on problem size, convergence rate, and the initial norm of $f(x)$.

Our results demonstrate that AATGS achieves performance comparable to the standard AA method with equivalent window sizes when applied to highly non-symmetric and nonlinear problems.
Furthermore, because of the short-term recurrence incorporated in AATGS, it outperforms AA on problems that are close to symmetric linear, even with a much smaller window size.
These experiments illustrate the properties of AATGS shown in the previous sections.
We also demonstrate the effectiveness of the restarting strategy.
Although it is possible to carefully tune the parameters and generate competitive results using standard AA, the proposed auto-restart AATGS has the advantage of not requiring the selection of the restart dimension. 

All of the methods were implemented in MATLAB 2023a.
We implemented AA by solving the least-square problem shown in Equation~\ref{eq:thetaj} using the pseudo-inverse with MATLAB's \texttt{pinv} function.
All experiments were conducted on the $\tt{Agate}$ cluster at the Minnesota Supercomputing Institute. 
The computing node features 64 GB of memory and is equipped with two sockets, each having a 2.45GHz AMD EPYC 7763 64-Core Processor.

\subsection{Bratu Problem}
In our first experiment, we solve a problem with a low degree of non-linearity  to demonstrate the benefits of the short-term recurrence in AATGS.
We consider a finite difference discretization of the following \textit{modified Bratu problem}~\cite{hajipour2018accurate} with Dirichlet boundary condition:
\begin{eqnarray}
    \Delta u + \alpha  u_x + \lambda e^u &=& 0 \ \ \text{in} \ \Omega, \nonumber \\
    u & = & 0 \ \ \text{on} \ \partial \Omega,
\end{eqnarray}
where $\Omega = [0,1]^2$.
We use \textit{centered finite differences}~\cite{chaudhry2008open,folland1995introduction,wilmott1995mathematics} to discretize the equation on a $202\times 202$ grid (including boundary points).
For our boundary value problem, this discretization results in a system of nonlinear equations with $n=200\times 200=40,000$ unknowns of the form:
\begin{equation}
    f(v) = Av + h \cdot \alpha Bv + h^2 \cdot \lambda \exp(v) = 0,
\end{equation}
where $v\in\mathbb{R}^n$ is the numerical solution at $n$ interior grid points, $h=1/201$ is the mesh size, $A\in\mathbb{R}^{n\times n}$ is a symmetric matrix, and $B\in\mathbb{R}^{n\times n}$ is a skew-symmetric matrix.
The fixed point iteration takes the form:
\begin{equation}
    g(v) = v + \beta f(v) = v + \beta(A v + h \cdot \alpha Bv + h^2 \cdot \lambda \exp(v)).
\end{equation}
The parameter $\lambda$ in the equation influences the change rate in the Jacobian of the problem.
Denoting by $J(v)$ the Jacobian at $v$, we have
\begin{equation}
    \|J(v_{j+1}) - J(v_j)\|_{\max} \leq h^2 \cdot \lambda\|\exp(v_{j+1}) - \exp(v_j)\|_{\infty},
\end{equation}
where $\|\cdot\|_{\max}$ is the matrix max norm.
This indicates that a larger $\lambda$ can potentially increase the non-linearity of the problem.
We take $\lambda=1$ in all of our experiments so that the equation is physically meaningful. 
In this case, the Jacobian's variation is limited, resulting in an almost linear problem.
The parameter $\alpha$ controls the degree of symmetry of the problem.
We test both symmetric ($\alpha = 0$) and non-symmetric ($\alpha \ne 0$) cases.

In all our experiments on Bratu problem, we use the zero vector as the initial solution and set the parameter $\beta=1.0$ for both AATGS and AA.
For comparison, we also include the results of fixed point iteration with $\beta=0.1$.

We begin our experiments with the symmetric case where $\alpha=0$.
To highlight the benefits of AATGS's short-term recurrence, we compare AATGS(3) with AA(20) and AA(100), and disable the restart for both AATGS and AA.
The left panel of Figure~\ref{fig:bratu} plots the iteration number versus the residual norm $\Vert f(v)\Vert_2$.
We can observe from the figure that AATGS performs better than AA in this experiment, even with a much smaller window size.
This is because when $\lambda=1$, the problem is close to a symmetric linear problem. 
In this case, AATGS(3) behaves similarly to AATGS($\infty$), which is equivalent to AA($\infty$).
This explains its superior performance compared to AA(20) and AA(100), demonstrating the potential advantage of AATGS over AA in handling nearly linear and symmetric problems. 

Next, we set $\alpha$ to 20 and solve a non-symmetric problem, noting that it remains nearly linear.
We first study the performance of AATGS with different restart strategies: no restart, fixed restart with dimension $50$, and auto-restart with $\eta=10^{3}$.
Since the problem is no longer symmetric, we slightly increase the window size to $m=5$.
We can observe from the middle panel of Figure~\ref{fig:bratu} that the AATGS(5) without restart underperforms the other two options.
The two restart versions have similar performance and the auto-restart is slightly better in this experiment.
This shows the importance of restart strategies.
As restart strategies can be very useful, in the following tests, we enable restart strategies for both AATGS and AA.

Finally, we compare the performance of AATGS and AA for the same non-symmetric problem with $\alpha=20$.
We compare AATGS(5) with auto-restart ($\eta=10^{3}$) against AA(5) and AA(20), both with a fixed restart dimension of $50$.
The results, shown in the right panel of Figure~\ref{fig:bratu}, demonstrate that AATGS(5) outperforms AA(5) and shows results comparable to those of AA(20).
This indicates that AATGS constructs a more effective subspace than standard AA even when the Jacobian is not symmetric.

\begin{figure}[tb]
     \centering
     \begin{subfigure}[b]{0.32\textwidth}
         \centering
         \includegraphics[width=\textwidth]{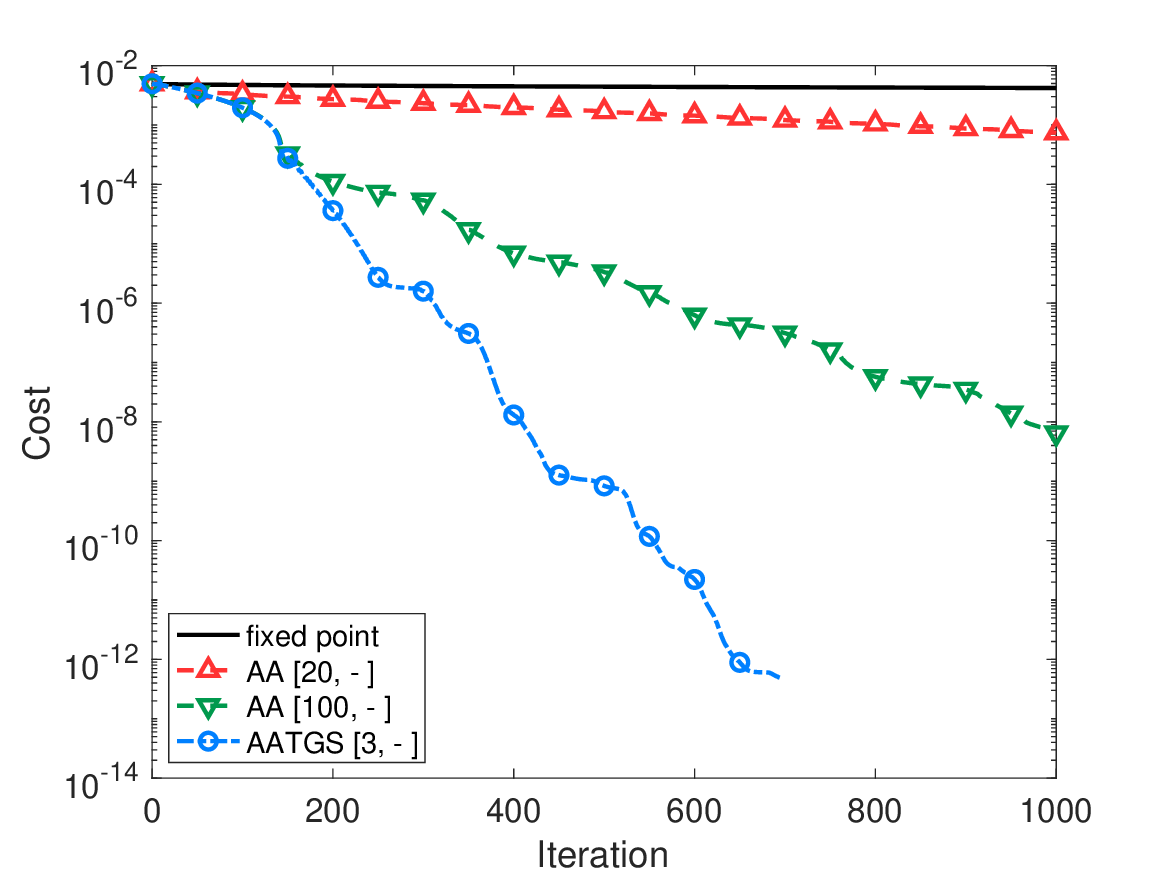}
         \label{fig:bratu_1}
     \end{subfigure}
     \hfill
     \begin{subfigure}[b]{0.32\textwidth}
         \centering
         \includegraphics[width=\textwidth]{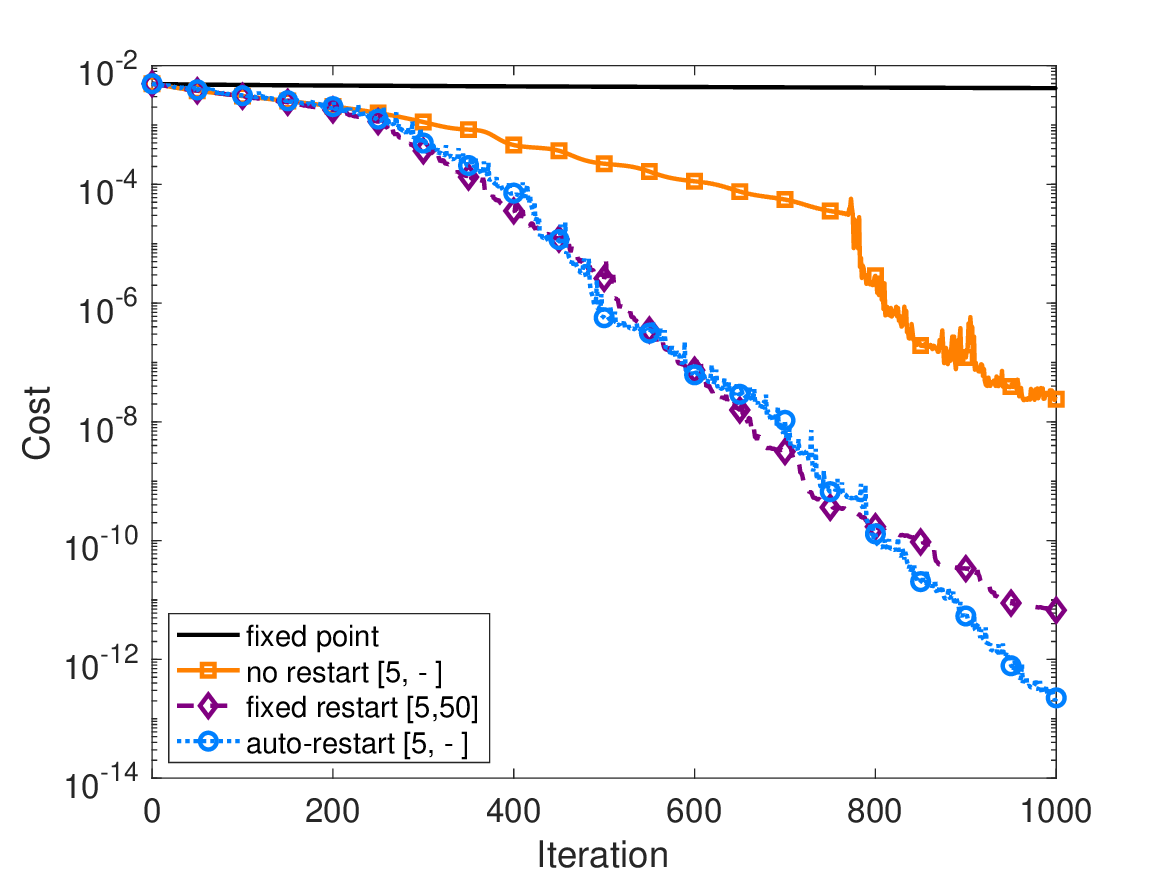}
         \label{fig:bratu_2}
     \end{subfigure}
     \hfill
     \begin{subfigure}[b]{0.32\textwidth}
         \centering
         \includegraphics[width=\textwidth]{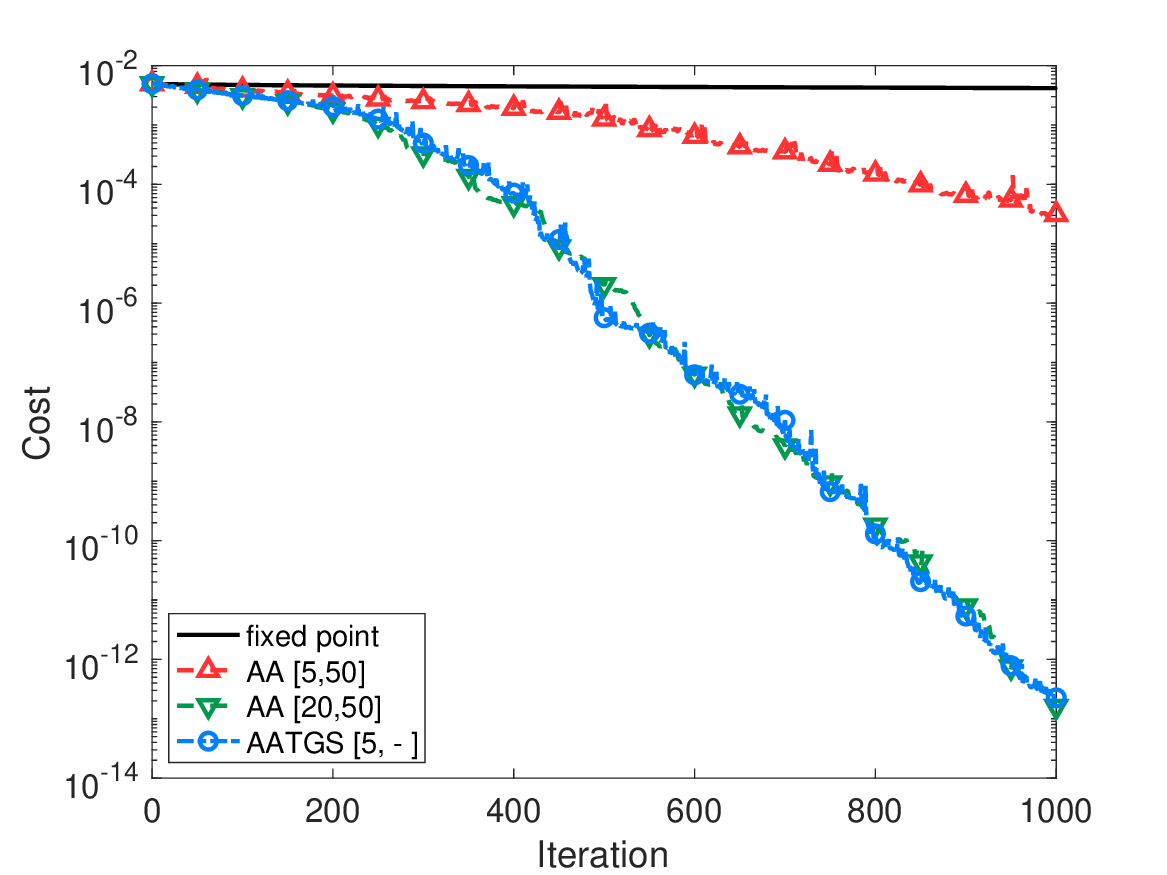}
         \label{fig:bratu_3}
     \end{subfigure}
        \caption{Bratu problem with initial solution $v_0 = 0$ and $\lambda=1$. (left) AATGS and AA with no restart for symmetric Jacobian with $\alpha=0$; (middle) AATGS with no restart, a fixed restart, and auto-restart for the non-symmetric Jacobian case. (right) AATGS with auto-restart and AA with a fixed restart for non-symmetric Jacobian with $\alpha=20$. $x$-axis is the iteration number and $y$-axis is the residual norm $\|f(v)\|_2$. Here, $[\cdot, \cdot]$ indicates the window size and the restart dimension of each method.}
        \label{fig:bratu}
\end{figure}

\subsection{Chandrasekhar's H-equation}
Next, we evaluate our method for Chandrasekhar's H-equation~\cite{ctk1995iter}. 
A form of the equation can be written as:
\begin{equation}
    {H}(\mu) - \left(
    1 - \frac{\omega}{2}\int_0^1\frac{\mu{H}(\nu)}{\mu + \nu} d\nu
    \right)^{-1} = 0,
    \label{eq:hequation}
\end{equation}
where $\omega \in [0,1]$ is a parameter, and we seek a solution ${H}\in C[0,1]$.
We discretize \eqref{eq:hequation} on a uniform grid and obtain the following discretized problem \cite{ctk1995iter}:
\begin{equation}
    [f(h)]_i := h_i - \left(1 - \frac{\omega}{2n}\sum_{j=1}^n \frac{\mu_i h_j}{\mu_i + \mu_j}\right)^{-1}
\end{equation}
where $h \in \mathbb{R}^n$ is the numerical solution at $n$ grid points, $\mu_i = \frac{i - 0.5}{n}$ for $1\leq i \leq n$, and the component-wise expression of the corresponding fixed point iteration $h=g(h)$ is given by
\begin{equation}
    [g(h)]_i = h_i + \beta [f(h)]_i = h_i + \beta\left[h_i - \left(1 - \frac{\omega}{2n}\sum_{j=1}^n \frac{\mu_i h_j}{\mu_i + \mu_j}\right)^{-1}\right].
\end{equation}
It is known that the Jacobian in this problem is non-symmetric~\cite{lin2021explicit}, as indicated by
its expression: 
\begin{equation}
    [J(h)]_{ik} = \delta_{ik} - \frac{\omega}{2n} \cdot \frac{\mu_i}{\mu_i + \mu_k} \cdot \left(1 - \frac{\omega}{2n}\sum_{j=1}^n \frac{\mu_i h_j}{\mu_i + \mu_j}\right)^{-2},
\end{equation}
where $\delta_{ik} = 1$ if $i = k$ and 0 otherwise.
The choice of $\omega$ can have an impact on the convergence of solution algorithms~\cite{wei2023convergence}. 
In our experiments, we set $n = 1,000$ and consider cases with $\omega = 0.99$ and
$\omega = 1.0$, both of which require careful timing for restarts in AA and
AATGS.

In this group of experiments, we use the vector of all ones as the initial solution and again set the parameter $\beta=1.0$ for both AATGS and AA.
Since the problem size is much smaller, we apply a smaller fixed restart dimension of $20$ for AA.
We compare AATGS and AA with window sizes $m=5$ and $m=20$ and again include results for fixed point iteration with $\beta=0.1$.
In this problem, a larger $m$ does not necessarily yield faster convergence, as observed from Figure \ref{fig:heq} that AA(5) consistently outperforms AA(20). Furthermore, we can see that AA(20) stagnates before a restart is triggered at step 20, which demonstrates the usefulness of the restarting procedure in this problem. With auto-restart, AATGS makes a stable selection of the
 window size, as shown by the identical performance of AATGS(5) and AATGS(20) in both figures.

 It is worth noting that a larger $\omega$ leads to a more challenging
 problem. When $\omega = 0.5$, the trajectories of AA(5), AA(20), AATGS(5), and
 AATGS(20) all overlap. However, when $\omega$ increases to 0.99, AA(20) fails
 to catch up with the other methods. When $\omega = 1.0$, AATGS outperforms AA.  This
 enhanced robustness of AATGS in dealing with numerical stability issues in the sequence of $x_j$'s 
 can also be attributed to the auto-restart strategy.

\begin{figure}[tb]
     \centering
     \begin{subfigure}[b]{0.45\textwidth}
         \centering
         \includegraphics[width=\textwidth]{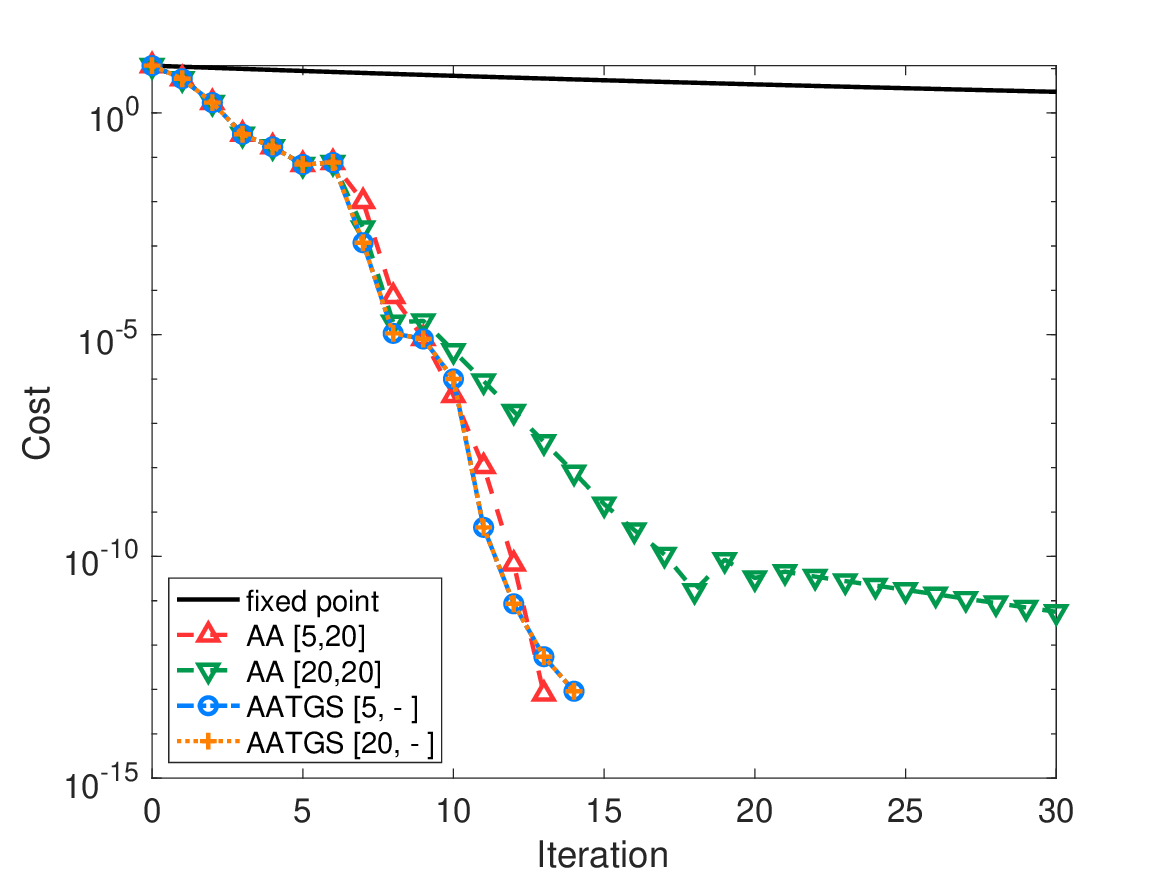}
         \label{fig:heq_1}
     \end{subfigure}
     \hfill
     \begin{subfigure}[b]{0.45\textwidth}
         \centering
         \includegraphics[width=\textwidth]{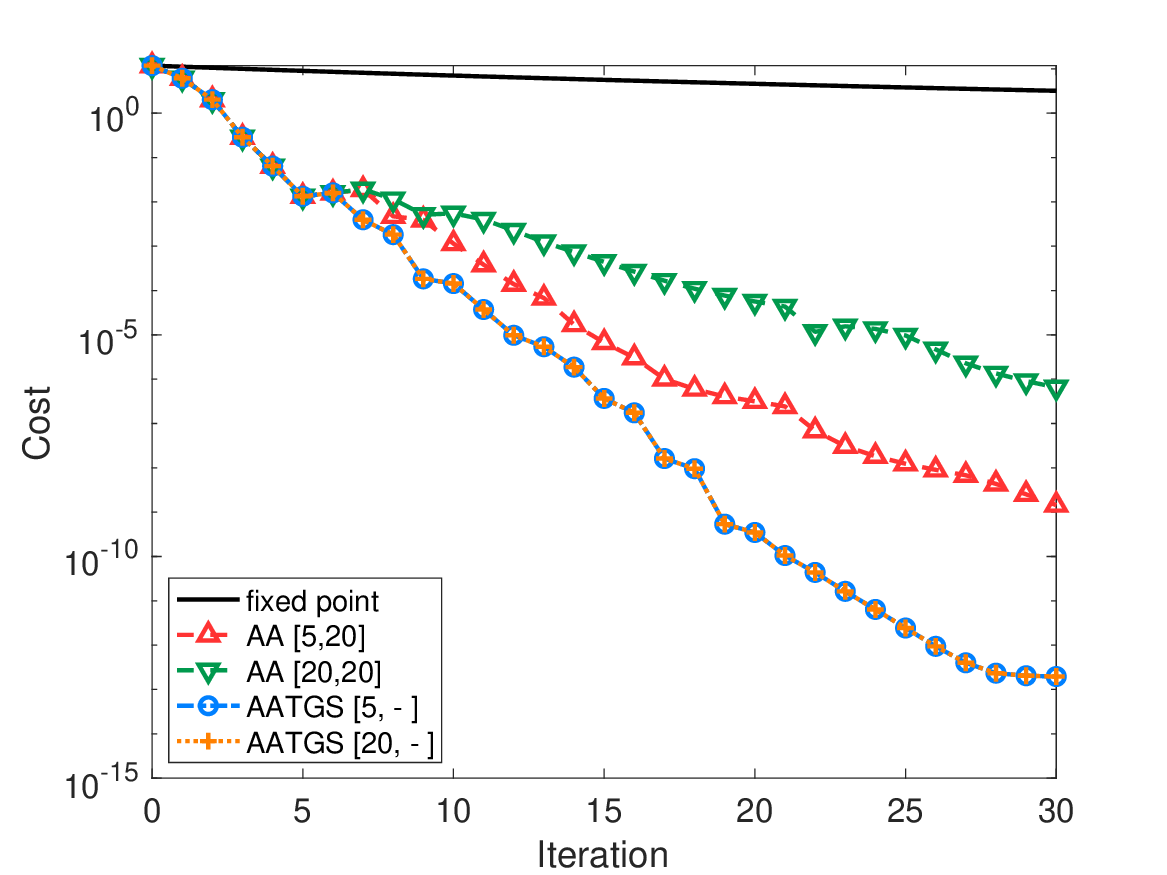}
         \label{fig:heq_2}
     \end{subfigure}
        \caption{Chandrasekhar's H-equation with dimension $n=1,000$. (left) The simpler case with $\omega=0.99$; (right) The harder case with $\omega=1.0$. $x$-axis is the iteration number and $y$-axis is the residual norm $\|f(h)\|_2$. Here, $[\cdot, \cdot]$ indicates the window size and the restart dimension of each method.}
        \label{fig:heq}
\end{figure}

\subsection{Lennard-Jones problem}

Next, we evaluate the performance of AATGS when solving the unconstrained minimization problem of the form
\begin{equation}
	\min_x \phi(x).
\end{equation}
We define $f(x)=-\nabla\phi(x)$, and write the fixed point iteration in the gradient descent form
\begin{equation}
    g(x) = x + \beta f(x) = x + \beta (-\nabla\phi(x)).
\end{equation}
Specifically, we optimize the geometry of molecules to achieve a minimum
total Lennard-Jones (LJ) potential energy. The LJ potential is defined as
follows\footnote{Thanks: We benefited from Stefan Goedecker's course site at
  Basel University.}:
\begin{equation}
    E(Y) = \sum_{i=1}^{N} \sum_{j=1}^{i-1} 4\epsilon
    \left[\left(\frac{\delta}{\|Y_{i,:} - Y_{j,:}\|}\right)^{12} - \left(\frac{\delta}{\|Y_{i,:} - Y_{j,:}\|}\right)^{6} \right].
\end{equation}
In this formulation, $N$ is the number of atoms, $\epsilon$ represents the well depth, $\delta$ is the distance between two non-bonding particles, and $Y\in\mathbb{R}^{N\times 3}$ with its $i$-th row $Y_{i,:}$ representing the coordinates of atom $i$.
We reformulate the problem by reshaping $Y$ into $x\in\mathbb{R}^{3N}$ where $[x_{3i-2},x_{3i-1},x_{3i}]=Y_{i,:}$ and defining the loss function $\phi(x) = E(Y)$.
In our experiments, we set both $\epsilon$ and $\delta$ to $1$ and simulate an Argon cluster starting from a perturbed initial Face-Centered-Cubic (FCC) structure \cite{meyer1964new}. 
We took 3 cells per direction, resulting in 27 unit cells. 
Given that each FCC cell includes 4 atoms, there are $N=108$ atoms in total. 
The challenge in this problem arises from the high exponents in the potential.

Figure \ref{fig:LJ} (left) shows an illustration of the geometry optimization in this problem, where the initial positions of the atoms are shown as blue dots, and the red triangles indicate the optimized final positions, which represent a local minimum around the initial positions rather than a global optimum.
We take $\beta=1.5\times10^{-4}$ in our experiments for both AATGS and AA.
Given that this is an unconstrained optimization problem, the Jacobian of
$\nabla \phi(x)$ is also the Hessian of $\phi(x)$, which is always symmetric. Therefore, we set the window size of AATGS to $m=3$. 
In Figure \ref{fig:LJ} (right), we compare AATGS against standard AA in three configurations. 
We can see that AATGS with a window size of $m=3$ and auto-restart strategy outperforms others. 
AA with $m=20$ and a restart dimension of 100 performs similarly to AA with $m=3$ and restart 10, and both surpass AA with $m=3$ and a restart dimension of 100. 
It again demonstrates the usefulness of the auto-restart strategy in AATGS for a non-trivial optimization problem.

\begin{figure}[tb]
     \centering
     \begin{subfigure}[b]{0.45\textwidth}
         \centering
         \includegraphics[width=\textwidth]{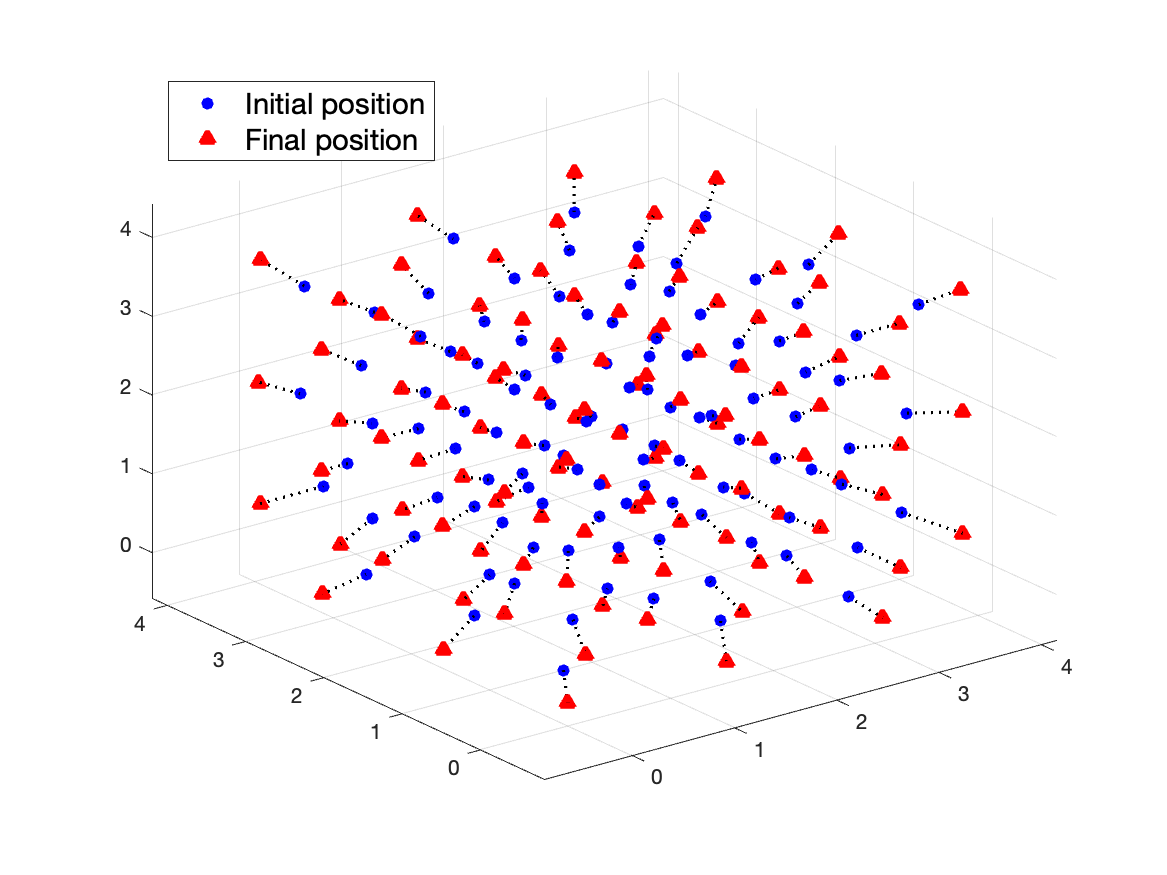}
         \label{fig:LJ_1}
     \end{subfigure}
     \hfill
     \begin{subfigure}[b]{0.45\textwidth}
         \centering
         \includegraphics[width=\textwidth]{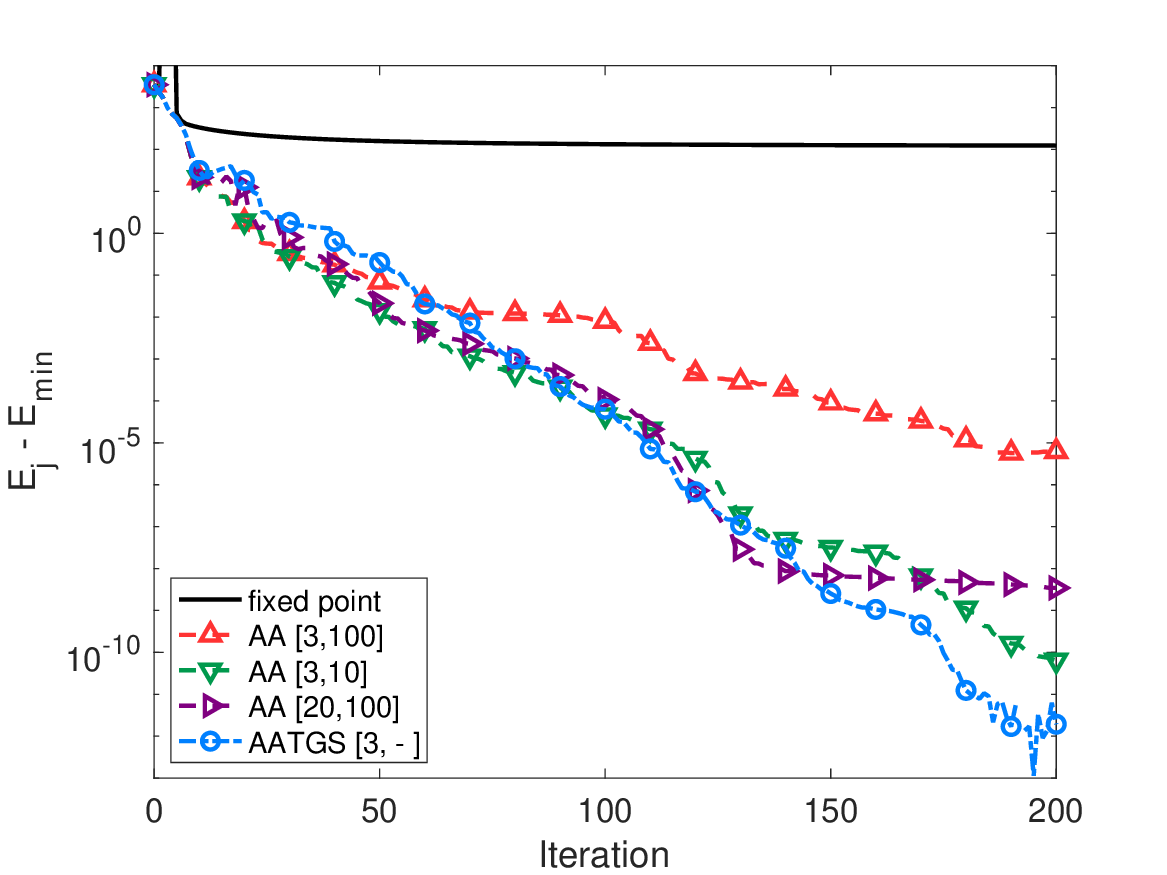}
         \label{fig:LJ_2}
     \end{subfigure}
        \caption{The Lennard-Jones problem. (left) The geometry of particles at the initial state and the final state; (right) The results of various methods in this experiment. $x$-axis is the iteration number and $y$-axis is the shifted energy $E_j - E_{\min}$. Note that, $E_{\min}$ is the minimum energy achieved by all considered methods so that the shifted energy is always positive. $[\cdot, \cdot]$ indicates the window size and the restart dimension of each method.}
        \label{fig:LJ}
\end{figure}

\subsection{Steady Navier–Stokes equations}
In our next experiment, we aim to solve a 2D lid-driven cavity problem described
by the steady Navier-Stokes equations (NSEs):
\begin{equation}
\begin{aligned}
    u\cdot\nabla u + \nabla p - 
    Re^{-1}\Delta u &= f, \\ 
    \quad \nabla\cdot u &= 0,
\end{aligned}
\end{equation}
with the domain $\Omega=(0,1)^2$ and the Dirichlet boundary condition
$(u,p) = (0,0)$ on the sides and bottom and $(1,0)$ on the lid.  Following the
settings in \cite{pollock2023filtering}, we set the Reynolds number $Re=10,000$
and use an initial guess of all zeros \footnote{Thanks: We would like to thank Sara Pollack and Leo G. Rebholz for sharing their 2D Steady Navier-Stokes equation codes with us.}. The discretization results in a problem
of size 190,643. Readers can refer to \cite{pollock2023filtering} for details of
the mesh. The fixed point iteration used by both AATGS and AA takes the form:
\begin{equation}
    g(v) = v + \beta f(v) = v + \beta (h(v) - v),
\end{equation}
where $v$ is the discretization of $(u,p)$ on grid points, and $h(v)$ performs one step of Picard iteration which maps $v$ to some specific approximate solution. Details on $h(v)$ can be found in \cite{doi:10.1137/18M1206151}.

The results in Figure \ref{fig:NS} compare Picard iterations, AA with window sizes $m=5$ and $m=10$, and AATGS with a window size of $m=5$. A restart is not necessary in this experiment since we can observe that both AA and AATGS converge without stagnation.
We also observe that Picard iteration fails to converge, which is likely due to the
extremely large Reynolds number. Both AATGS and AA manage to converge at a
similar rate. Given the non-symmetric and nonlinear nature of this problem, we cannot
expect significant gains from AATGS over AA in this case. Indeed, the methods behave similarly.

\begin{figure}[tb]
     \centering
     \begin{subfigure}[b]{0.45\linewidth}
         \centering
         \includegraphics[width=\textwidth]{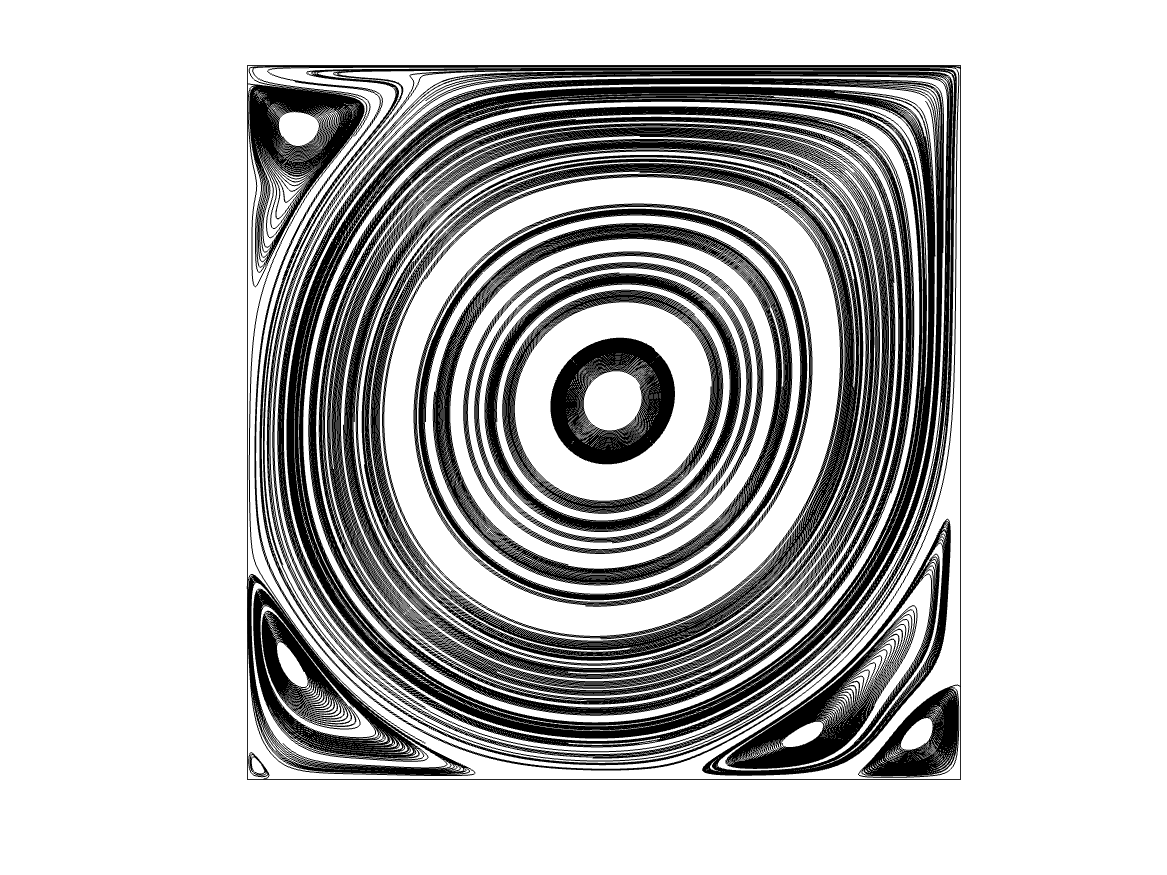}
         \label{fig:NS_1}
     \end{subfigure}
     \hfill
     \begin{subfigure}[b]{0.45\linewidth}
         \centering
         \includegraphics[width=\textwidth]{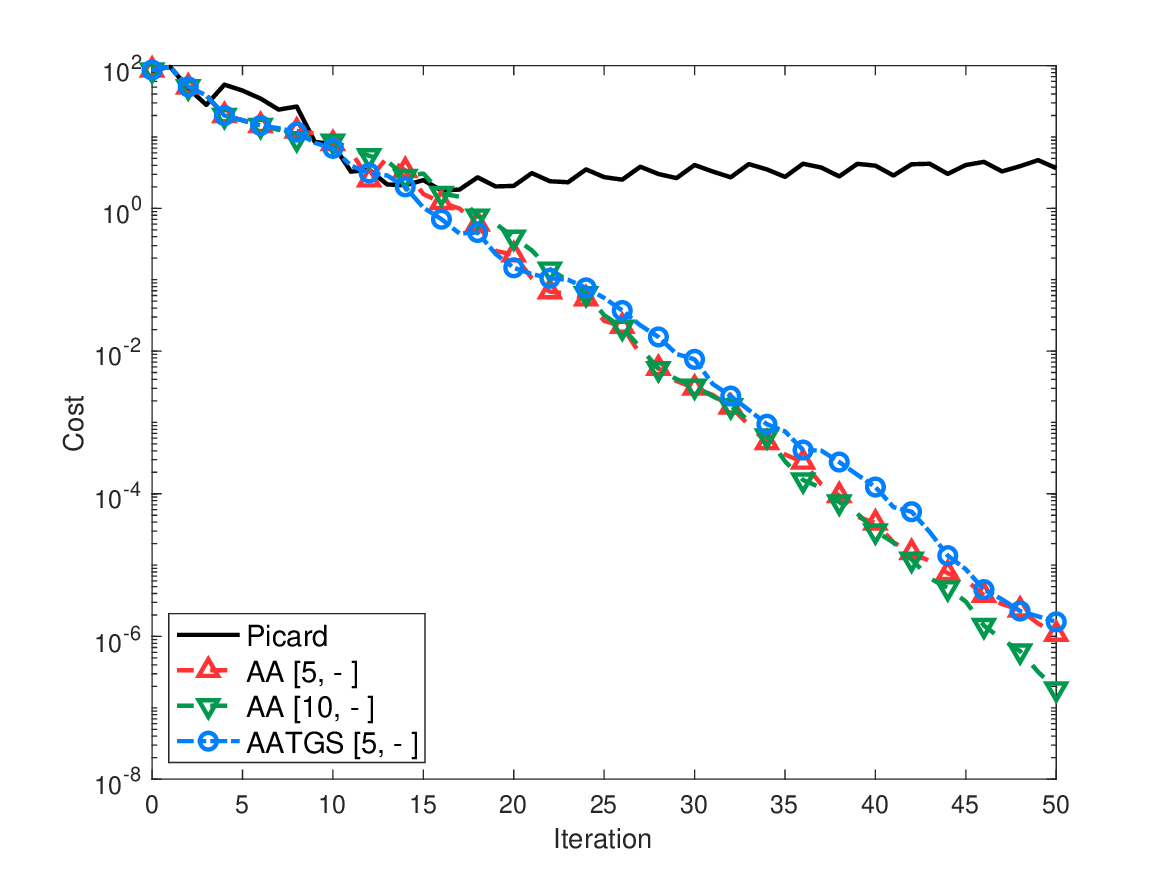}
         \label{fig:NS_2}
     \end{subfigure}
        \caption{2D Steady Navier-Stokes equations with the Reynolds number $Re=10,000$. (left) The streamlines of the solution given by AATGS at step 50; (right) The results of various methods in this experiment. $x$-axis is the iteration number and $y$-axis is the residual norm of $\|\text{Picard}(v) - v\|_2$. $[\cdot, \cdot]$ indicates the window size and the restart dimension of each method.}
        \label{fig:NS}
\end{figure}

\subsection{Regularized Logistic Regression}
Regularized logistic regression is a powerful tool for binary classification tasks, particularly when dealing with datasets that have a large number of features. In this experiment, we investigate the application of regularized logistic regression to the Madelon dataset \footnote{\url{https://archive.ics.uci.edu/dataset/171/madelon}}. The training set consists of $N=2,000$ samples and $n=500$ features. The objective can be formulated as follows:
\begin{equation}
    \min_\theta ~ \frac{1}{N} \sum_{i=1}^{N} \log(1 + \exp(- y_i \cdot x_i^\top\theta)) + \frac{\lambda}{2}\|\theta\|_2^2,
\end{equation}
where $x_i$ represents the feature vector of the $i$-th sample (each feature is normalized to have a mean of 0 and a standard deviation of 1 across all samples), $y_i$ represents the label of the $i$-th sample (either -1 or 1 for binary classification), $\theta \in \mathbb{R}^n$ is the parameter vector to be optimized, $\lambda$ is the regularization parameter that controls the balance between fitting the training data well and preventing overfitting by penalizing large parameter values. 

Figure \ref{fig:reglog} illustrates the shifted training loss as a function of the iteration number. We set the fixed point iteration parameter $\beta = 1.0$, the regularization parameter $\lambda = 0.01$, and the window size $m=3$. We use the zero vector as the initial solution. In this comparison, we focus on AATGS with varying auto-restart threshold $\eta$ ranging from $10^1$ to $\infty$. The results demonstrate the efficacy and simplicity of parameter tuning for our auto-restart strategy, as the loss curves for $\eta = 10^2$ to $10^6$ show small variance. The performance deteriorates only when $\eta = 10^1$ -- resulting in excessive, redundant restarts -- and when $\eta = \infty$ -- leads to the absence of restarts. Through our testing across many experiments, the default setting of $\eta = 10^3$ often delivers a sufficiently accurate solution.

\begin{figure}[tb]
         \centering
         \includegraphics[width=0.5\linewidth]{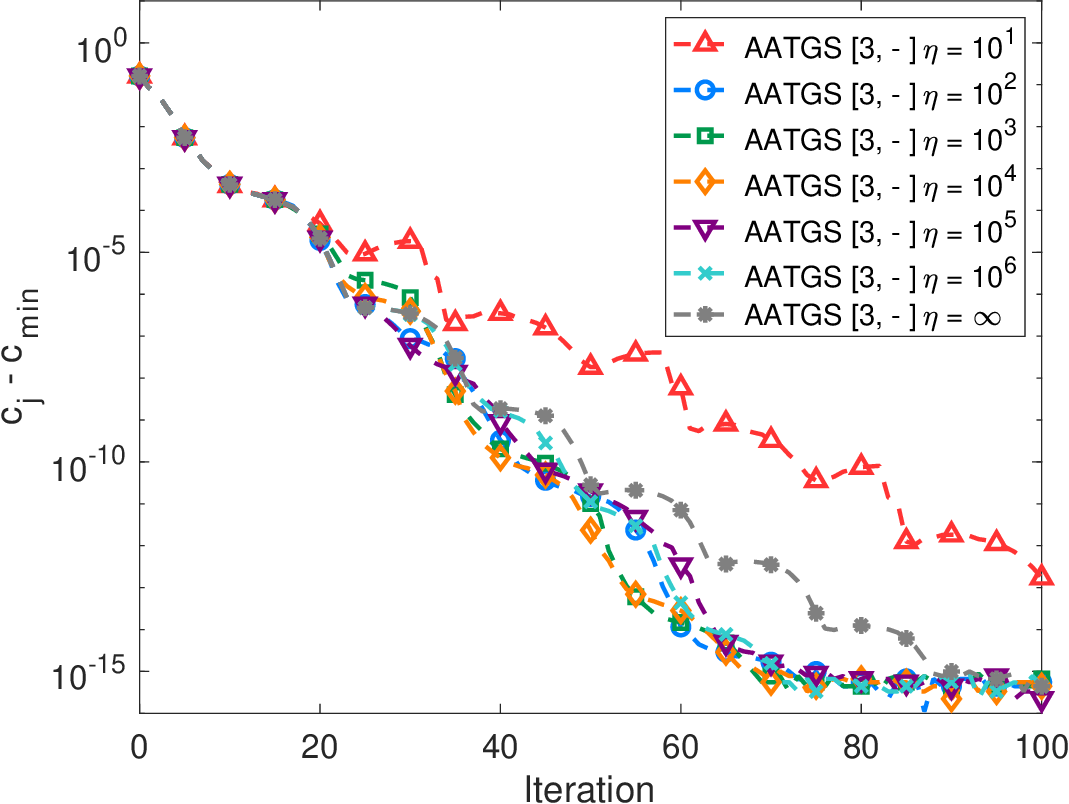}
        \caption{The results of various $\eta$'s for the regularized logistic regression on the Madelon dataset. $x$-axis is the iteration number and $y$-axis is the shifted training loss $c_j - c_{\min}$. Note that, $c_{\min}$ is the minimum training loss achieved by all considered methods so that the shifted loss is always positive. $[\cdot, \cdot]$ indicates the window size and the restart dimension of each method.}
        \label{fig:reglog}
\end{figure}

In Table \ref{tab:eta}, we present the number of iterations (up to 1000) required for AATGS to achieve a relative loss smaller than $10^{-12}$. The regularization parameter $\lambda$ varies from $10^{0}$ to $10^{-5}$, changing the optimization problems from relatively simple to significantly difficult to solve. It is important to note that our goal in these comparisons is not to achieve the highest accuracy but rather to elucidate the characteristics of AATGS. With a window size of $m=3$, it is observed that as the problem becomes more challenging (with smaller $\lambda$), the number of required iterations generally increases. However, AATGS with $\eta = 10^3$ to $10^5$ always exhibits similar performance. Only extremely high or low $\eta$'s tend to be significantly slower than other values and fail to converge within 1000 iterations. This further confirms that $\eta$ offers a broad selection range. 

\begin{table}[htb]
\centering
\begin{tabular}{c|cccccc}
    \toprule
    \multirow{2}{*}{$\lambda$} & \multicolumn{6}{c}{Number of Iterations}  \\
    & $\eta=10^{1}$ & $\eta=10^{2}$ & $\eta=10^3$ & $\eta=10^4$ & $\eta=10^5$ & $\eta=\infty$ \\
    \midrule
    $10^{0}$ & 21 & 20 & 22 & 22 & 22 & 22 \\
    $10^{-1}$ & 52 & 50 & 48 & 51 & 51 & 56 \\
    $10^{-2}$ & 200 & 113 & 105 & 117 & 113 & 167 \\
    $10^{-3}$ & F & F & 188 & 173 & 201 & 418 \\
    $10^{-4}$ & F & 473 & 251 & 209 & 228 & F \\
    $10^{-5}$ & F & F & 254 & 228 & 251 & F \\
    \bottomrule
\end{tabular}
\caption{A comparison of AATGS with a fixed window size $m=3$ across various auto-restart thresholds $\eta$ (columns) and regularization parameters $\lambda$ (rows) is presented. This table displays the number of iterations required for AATGS to achieve a relative loss smaller than $10^{-12}$. The notation `F' indicates cases where the method fails to converge within 1000 iterations.}
\label{tab:eta}
\end{table}

\subsection{Minimax Optimization}\label{sec:minimax}

Bilinear games are often regarded as an important example of understanding new
algorithms and techniques for solving general minimax problems
\cite{DBLP:conf/iclr/GidelBVVL19, gdaam}.  In this experiment, we study the
following zero-sum bilinear games:
\begin{equation}
\label{eq:bilinear}
    \min_{{x}\in \mathbb{R}^n}\max_{{y}\in \mathbb{R}^n}\phi({x},{y}) =  {x}^{T}{A}{y} + {b}^{T}{x} + {c}^T{y},
\end{equation}
where $A$ is a full-rank matrix.
The Nash equilibrium to the above problem is given by $({x}^{\ast},{y}^{\ast})=(-{A}^{-T}{c},-{A}^{-1}{b})$. 
We use the alternating Gradient Descent Ascent (GDA) algorithm to solve the problem in the following form:
{
\begin{equation}
\begin{aligned}
    \begin{bmatrix}
        {x}_{j+1} \\ {y}_{j+1}
    \end{bmatrix}
    &= 
    \begin{bmatrix}
        {x}_j \\ {y}_j
    \end{bmatrix}
    +
    \beta\cdot
    \begin{bmatrix}
        -\nabla_x \phi({x}_j,{y}_j) \\
        \nabla_y \phi({x}_{j+1},{y}_j)
    \end{bmatrix} \\
    &=
    \begin{bmatrix}
        {I} & -\beta {A} \\
        \beta {A}^T & {I} - \beta^2 {A}^T {A}
    \end{bmatrix}
    \begin{bmatrix}
        {x}_j \\ {y}_j
    \end{bmatrix}
    -\beta
    \begin{bmatrix}
        {b} \\ \beta{A}^T{b} - {c}
    \end{bmatrix}
\end{aligned}
\label{eq:minimax_sim}
\end{equation}
where the solution of the above fixed point iteration is the root of the following nonlinear equation $f$:
\begin{equation}
    f\left(
    \begin{bmatrix}
        {x} \\ {y}
    \end{bmatrix}
    \right)
    :=
    \begin{bmatrix}
        {0} & -{A} \\
        {A}^T & -\beta{A}^T {A}
    \end{bmatrix}
    \begin{bmatrix}
        {x} \\ {y}
    \end{bmatrix}
    -
    \begin{bmatrix}
        {b} \\ \beta{A}^T{b}-{c}
    \end{bmatrix} .
\label{eq:minimax_f2}
\end{equation}
The coefficients of the initial problem, ${A} \in \mathbb{R}^{100 \times 100}$, ${b} \in \mathbb{R}^{100}$, and ${c} \in \mathbb{R}^{100}$, are generated using random numbers following the distribution $\mathcal{N}(0,1)$. Subsequently, ${A}$ undergoes normalization to ensure its 2-norm equals $1$. The initial guess is also generated using random numbers following the distribution $\mathcal{N}(0,1)$. The cost of this problem is defined as the relative distance to the optimal solution, i.e., $c_j := \|({x}_j,{y}_j) - ({x}^*,{y}^*)\|_2 / \|({x}^*,{y}^*)\|_2$, where $({x}_j,{y}_j)$ is the iteration at step $j$.
}

\begin{figure}[tb]
    \centering
    \begin{subfigure}[t]{0.45\linewidth}
         \centering
         \includegraphics[width=\textwidth]{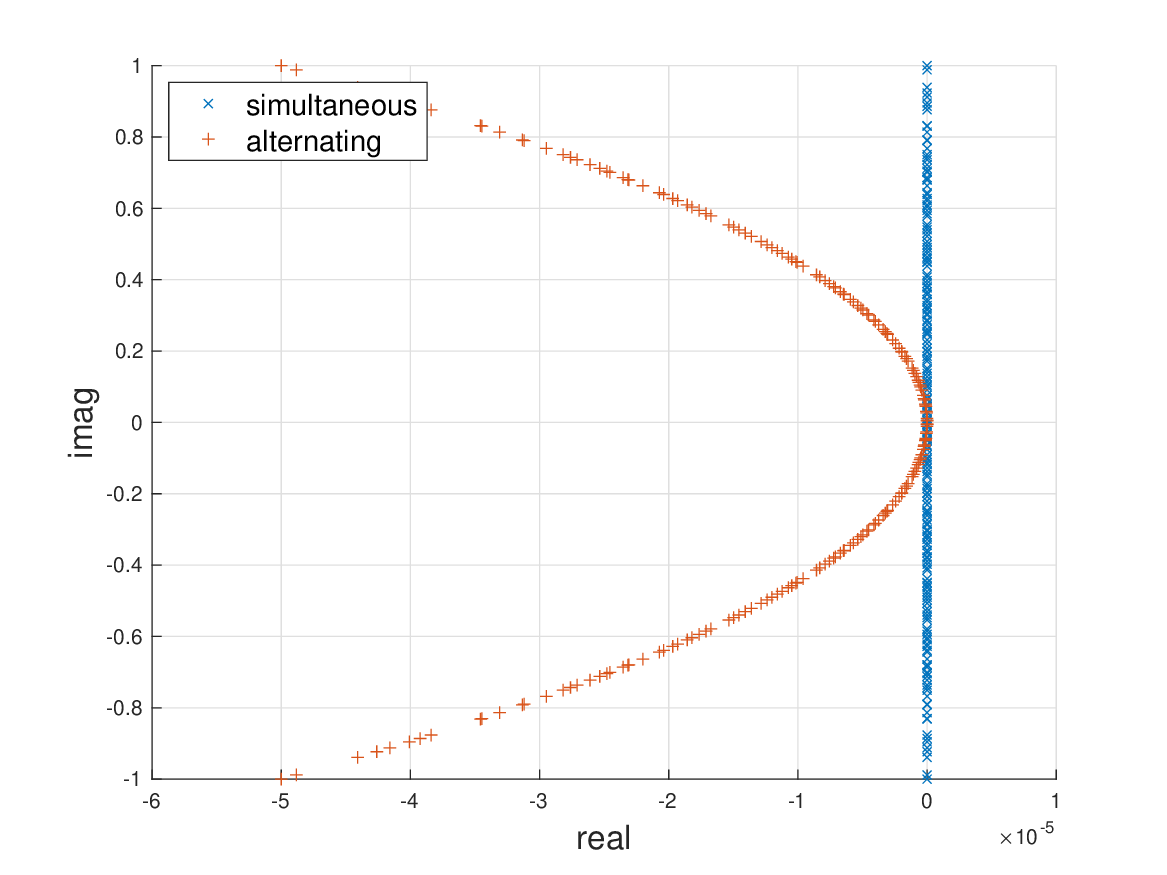}
         \label{fig:lr_1}
    \end{subfigure}
    \hfill
    \begin{subfigure}[t]{0.45\linewidth}
        \includegraphics[width=\linewidth]{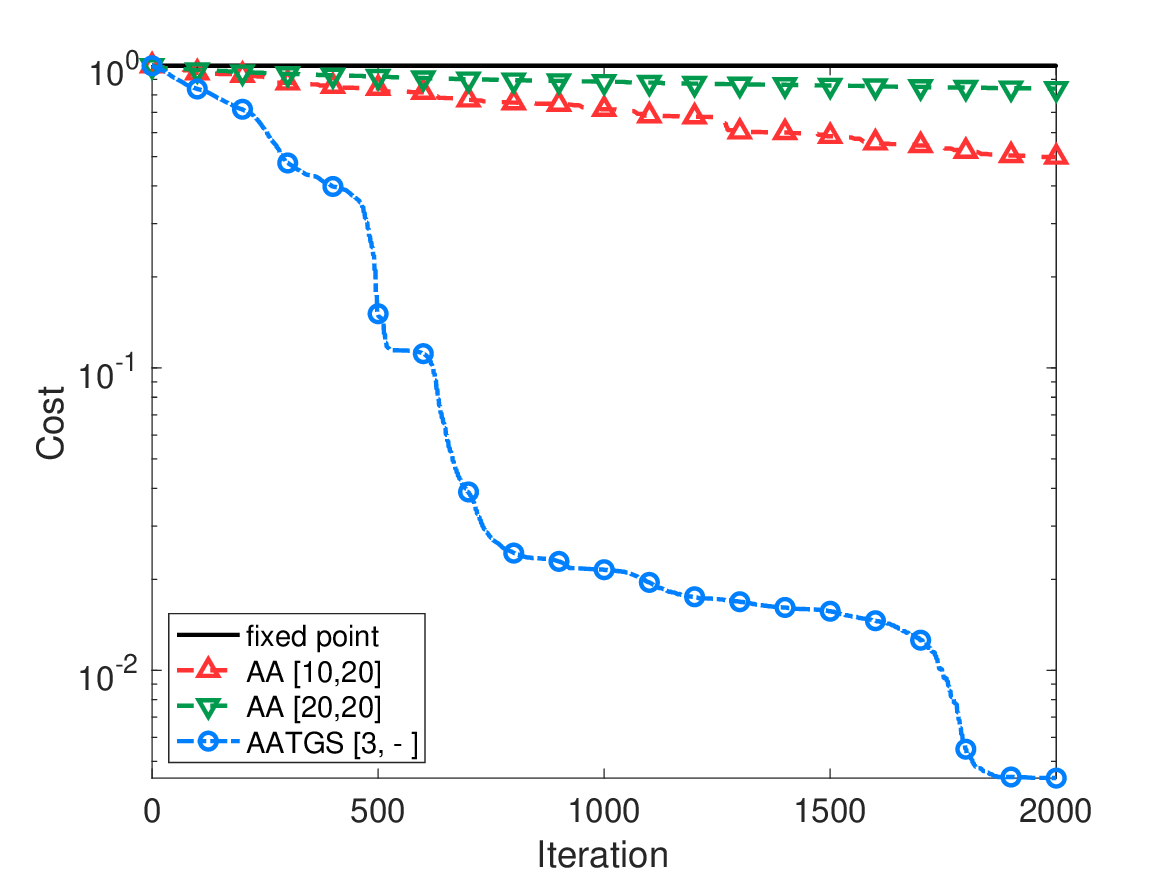}
        \label{fig:minimax1}
    \end{subfigure}
    \caption{Minimax optimization on a bilinear game. (left) Spectrum of the linear systems corresponding to the simultaneous GDA and alternating GDA. $x$-axis is the real part and $y$-axis is the imaginary part. Blue crosses represent the eigenvalues of the simultaneous GDA. Red plus signs represent the eigenvalues of the alternating GDA; (right) The results of various methods in this experiment. $x$-axis is the iteration number and $y$-axis is the relative Euclidean distance to the optimal solution. $[\cdot, \cdot]$ indicates the window size and the restart dimension of each method.}
    \label{fig:reglog_minimax}
\end{figure}

{Note that Equation \eqref{eq:minimax_sim} is referred to as the \textit{alternating} GDA since we update ${x}_{j+1}$ and ${y}_{j+1}$ in an alternating manner. For the \textit{simultaneous} GDA, we update ${x}_{j+1}$ and ${y}_{j+1}$ simultaneously. However, this leads to a skew-symmetric linear system to solve (e.g., consider $\beta=0$ in Equation \eqref{eq:minimax_f2}) which has numerical issues as mentioned at the end of Section \ref{sec:st_aatgs}. Furthermore, the difference in spectra is shown in Fig.~\ref{fig:reglog_minimax} (left), where the eigenvalues of the coefficient matrix in simultaneous GDA are purely imaginary, while those of the alternating GDA have small real parts. Therefore, we consider the alternating GDA in this experiment.}

 In Figure \ref{fig:reglog_minimax} (right), we compare AATGS with standard AA under
  different settings. Since the coefficient matrix defined in function $f$ in Equation \eqref{eq:minimax_f2} is skew-symmetric plus a
  symmetric perturbation, we expect a similar short-term recurrence in AATGS and
  therefore we set the window size to $m=3$. In addition, we employ the auto-restart
  strategy instead of a fixed restart. For the baseline methods, we consider AA with
  window sizes $m=10$ and $m=20$, along with a fixed restart dimension of
  20. 
  Note that we use a smaller restart dimension because both AA options fail to converge if we use a restart dimension of 50.
  Moreover, we set $\beta = 10^{-4}$ to ensure
  that all methods do not diverge in most cases. We observe that after 2000
  iterations, AATGS manages to converge with a relative distance of around
  $0.0044$, while the AAs still have relative distances of 0.69 and 0.84 from
  the optimal solution. This experiment illustrates the appealing behavior of
  AATGS in solving linear problems that are nearly skew-symmetric. 

\section{Conclusion}
\label{sec:conclusion}
This paper introduced what may be termed a `symmetric version' of Anderson Acceleration.
When the fixed point iteration handled by  Anderson Acceleration is a linear iteration,
then AA does not take advantage of symmetry in the case when the iteration matrix is also symmetric.
The Truncated Gram-Schmidt variant of AA (AATGS) introduced in this paper, addresses this issue.
AATGS  is mathematically equivalent to AA when 
the depth of both algorithms is $m = \infty$. However, when the problem is linear and symmetric,
AATGS($\infty$) simplifies in that only a few vectors must be saved instead of all of the previous
directions generated, in order to produce the same iterates as $AA(\infty)$.
This can lead to substantial savings in memory and computational requirements for large problems.
From a practical point of view, the original AATGS algorithm without any modification can suffer from numerical
stability issues. A careful restarting strategy was developed to restart when deemed necessary by a
simple short-term scalar recurrence designed to mimic the behavior of the numerical errors. Equipped with
this artifice, the algorithm showed good robustness, often outperforming the original AA at a lower cost.
This was confirmed by a few numerical experiments, with applications ranging
 from nonlinear partial differential equations to challenging optimization problems.
 The numerical experiments showed that for problems whose Jacobien is nearly symmetric and for
 optimization problems (Hessian is symmetric), AATGS can be vastly superior to AA and this is expected from
 theory.
 
In the future, we plan to explore the applicability and efficacy of AATGS when applied to 
stochastic optimization problems. We will also study the exploitation of information on 
the Jacobian during the iteration in order to improve both robustness and efficiency,  as done in
\cite{nltgcr}.

\section{Acknowledgements}

The authors acknowledge the Minnesota Supercomputing Institute \\ (MSI) at the University of Minnesota for providing resources that contributed to the research results reported within this paper (\url{http://www.msi.umn.edu}).

\bibliographystyle{siam}

\bibliography{aatgs}

\begin{thebibliography}{10}

\bibitem{Anderson65}
{\sc D.~G. Anderson}, {\em Iterative procedures for non-linear integral equations}, Assoc. Comput. Mach., 12 (1965), pp.~547--560.

\bibitem{doi:10.1137/20M132938X}
{\sc W.~Bian, X.~Chen, and C.~T. Kelley}, {\em Anderson acceleration for a class of nonsmooth fixed-point problems}, SIAM Journal on Scientific Computing, 43 (2021), pp.~S1--S20.

\bibitem{chaudhry2008open}
{\sc M.~H. Chaudhry et~al.}, {\em Open-channel flow}, vol.~523, Springer, 2008.

\bibitem{hans22}
{\sc H.~De~Sterck and Y.~He}, {\em Linear asymptotic convergence of anderson acceleration: Fixed-point analysis}, SIAM Journal on Matrix Analysis and Applications, 43 (2022), pp.~1755--1783.

\bibitem{de2021anderson}
{\sc H.~De~Sterck, Y.~He, and O.~A. Krzysik}, {\em Anderson acceleration as a krylov method with application to asymptotic convergence analysis}, arXiv preprint arXiv:2109.14181,  (2021).

\bibitem{Eis-Elm-Sch}
{\sc S.~C. Eisenstat, H.~C. Elman, and M.~H. Schultz}, {\em Variational iterative methods for nonsymmetric systems of linear equations}, SIAM Journal on Numerical Analysis, 20 (1983), pp.~345--357.

\bibitem{sara20}
{\sc C.~Evans, S.~Pollock, L.~G. Rebholz, and M.~Xiao}, {\em A proof that anderson acceleration improves the convergence rate in linearly converging fixed-point methods (but not in those converging quadratically)}, SIAM Journal on Numerical Analysis, 58 (2020), pp.~788--810.

\bibitem{eyert:acceleration96}
{\sc V.~Eyert}, {\em A comparative study on methods for convergence acceleration of iterative vector sequences}, J. Comput. Phys., 124 (1996), pp.~271--285.

\bibitem{folland1995introduction}
{\sc G.~B. Folland}, {\em Introduction to partial differential equations}, vol.~102, Princeton university press, 1995.

\bibitem{DBLP:conf/iclr/GidelBVVL19}
{\sc G.~Gidel, H.~Berard, G.~Vignoud, P.~Vincent, and S.~Lacoste{-}Julien}, {\em A variational inequality perspective on generative adversarial networks}, in 7th International Conference on Learning Representations, {ICLR}, 2019.

\bibitem{hajipour2018accurate}
{\sc M.~Hajipour, A.~Jajarmi, and D.~Baleanu}, {\em On the accurate discretization of a highly nonlinear boundary value problem}, Numerical Algorithms, 79 (2018), pp.~679--695.

\bibitem{nltgcr}
{\sc H.~He, Z.~Tang, S.~Zhao, Y.~Saad, and Y.~Xi}, {\em nltgcr: A class of nonlinear acceleration procedures based on conjugate residuals}, SIAM Journal on Matrix Analysis and Applications, 45 (2024), pp.~712--743.

\bibitem{gdaam}
{\sc H.~He, S.~Zhao, Y.~Xi, J.~C. Ho, and Y.~Saad}, {\em {GDA-AM:} on the effectiveness of solving min-imax optimization via anderson mixing}, in The Tenth International Conference on Learning Representations, {ICLR} 2022, Virtual Event, April 25-29, 2022, OpenReview.net, 2022.

\bibitem{ctk1995iter}
{\sc C.~T. Kelley}, {\em Iterative Methods for Linear and Nonlinear Equations}, Society for Industrial and Applied Mathematics, 1995.

\bibitem{lin2021explicit}
{\sc D.~Lin, H.~Ye, and Z.~Zhang}, {\em Explicit superlinear convergence rates of broyden's methods in nonlinear equations}, arXiv preprint arXiv:2109.01974,  (2021).

\bibitem{lupo2022anderson}
{\sc M.~Lupo~Pasini and M.~P. Laiu}, {\em Anderson acceleration with approximate calculations: Applications to scientific computing}, Numerical Linear Algebra with Applications,  (2022), p.~e2562.

\bibitem{10.5555/3524938.3525552}
{\sc V.~V. Mai and M.~Johansson}, {\em Anderson acceleration of proximal gradient methods}, in Proceedings of the 37th International Conference on Machine Learning, ICML'20, JMLR.org, 2020.

\bibitem{meyer1964new}
{\sc L.~Meyer, C.~Barrett, and P.~Haasen}, {\em New crystalline phase in solid argon and its solid solutions}, The Journal of Chemical Physics, 40 (1964), pp.~2744--2745.

\bibitem{10.1145/3197517.3201290}
{\sc Y.~Peng, B.~Deng, J.~Zhang, F.~Geng, W.~Qin, and L.~Liu}, {\em Anderson acceleration for geometry optimization and physics simulation}, ACM Trans. Graph., 37 (2018).

\bibitem{pollock2023filtering}
{\sc S.~Pollock and L.~G. Rebholz}, {\em Filtering for anderson acceleration}, SIAM Journal on Scientific Computing, 45 (2023), pp.~A1571--A1590.

\bibitem{doi:10.1137/18M1206151}
{\sc S.~Pollock, L.~G. Rebholz, and M.~Xiao}, {\em Anderson-accelerated convergence of picard iterations for incompressible navier--stokes equations}, SIAM Journal on Numerical Analysis, 57 (2019), pp.~615--637.

\bibitem{pul80}
{\sc P.~Pulay}, {\em Convergence acceleration of iterative sequences. the case of {SCF} iteration}, Chem. Phys. Lett., 73 (1980), pp.~393--398.

\bibitem{Pulay-DIIS}
\leavevmode\vrule height 2pt depth -1.6pt width 23pt, {\em Improved {SCF} convergence acceleration}, J. Comput. Chem., 3 (1982), pp.~556--560.

\bibitem{leo23}
{\sc L.~G. Rebholz and M.~Xiao}, {\em The effect of anderson acceleration on superlinear and sublinear convergence}, J. Sci. Comput., 96 (2023).

\bibitem{FangSaad07}
{\sc H.~ren Fang and Y.~Saad}, {\em Two classes of multisecant methods for nonlinear acceleration}, Numer Linear Algebra Appl., 16 (2009), pp.~197--221.

\bibitem{Saad-book2}
{\sc Y.~Saad}, {\em Iterative Methods for Sparse Linear Systems, 2nd edition}, SIAM, Philadelpha, PA, 2003.

\bibitem{Saad-Schultz-GMRES}
{\sc Y.~Saad and M.~H. Schultz}, {\em {GMRES:} a generalized minimal residual algorithm for solving nonsymmetric linear systems}, SIAM J. Sci. Stat. Comput., 7 (1986), pp.~856--869.

\bibitem{Swirydowicz_Langou_Ananthan_Yang_Thomas_2020}
{\sc K.~Swirydowicz, J.~Langou, S.~Ananthan, U.~Yang, and S.~Thomas}, {\em Low synchronization gram–schmidt and generalized minimal residual algorithms}, Numerical Linear Algebra with Applications, n/a (2020), p.~e2343.

\bibitem{Kelley15}
{\sc A.~Toth and C.~T. Kelley}, {\em Convergence analysis for anderson acceleration}, SIAM Journal on Numerical Analysis, 53 (2015), pp.~805--819.

\bibitem{homer11}
{\sc H.~F. Walker and P.~Ni}, {\em Anderson acceleration for fixed-point iterations}, SIAM Journal on Numerical Analysis, 49 (2011), pp.~1715--1735.

\bibitem{10.1007/s10915-021-01548-2}
{\sc D.~Wang, Y.~He, and H.~De~Sterck}, {\em On the asymptotic linear convergence speed of anderson acceleration applied to admm}, J. Sci. Comput., 88 (2021).

\bibitem{wei2021stochastic}
{\sc F.~Wei, C.~Bao, and Y.~Liu}, {\em Stochastic anderson mixing for nonconvex stochastic optimization}, in Advances in Neural Information Processing Systems, A.~Beygelzimer, Y.~Dauphin, P.~Liang, and J.~W. Vaughan, eds., 2021.

\bibitem{wei2023convergence}
{\sc F.~Wei, C.~Bao, Y.~Liu, and G.~Yang}, {\em Convergence analysis for restarted anderson mixing and beyond}, arXiv preprint arXiv:2307.02062,  (2023).

\bibitem{wilmott1995mathematics}
{\sc P.~Wilmott, S.~Howson, S.~Howison, J.~Dewynne, et~al.}, {\em The mathematics of financial derivatives: a student introduction}, Cambridge university press, 1995.

\bibitem{fei22}
{\sc F.~Xue}, {\em One-step convergence of inexact anderson acceleration for contractive and non-contractive mappings}, Electronic Transactions on Numerical Analysis, 55 (2022), pp.~285--309.

\end{thebibliography}

\end{document}